\documentclass[a4paper]{amsart}
\usepackage{graphicx,amsmath,amsfonts,latexsym,amssymb,amsthm,mathrsfs,color}
\usepackage[latin1]{inputenc}
\newcommand{\B}{{\mathbb B}}
\newcommand{\C}{{\mathbb C}}

\newcommand{\N}{{\mathbb N}}

\newcommand{\R}{{\mathbb R}}

\newcommand{\cD}{{\mathcal D}}
\newcommand{\cE}{{\mathcal E}}

\newcommand{\cH}{{\mathcal H}}

\newcommand{\cL}{{\mathcal L}}

\newcommand{\cP}{{\mathcal P}}

\newcommand{\cS}{{\mathcal S}}

\newcommand{\ga}{\alpha}

\renewcommand{\gg}{\gamma}

\newcommand{\gd}{\delta}
\newcommand{\gD}{\Delta}

\newcommand{\gve}{\varepsilon}

\newcommand{\gl}{\lambda}

\newcommand{\go}{\omega}

\newcommand{\gt}{\theta}

\newcommand{\gs}{\sigma}
\newcommand{\Proof}[1]{{\em Proof}. #1~\hfill$\Box$\medskip}
\newcommand{\skp}[1]{\langle#1\rangle}
\renewcommand{\Re}{\mathop{\rm Re}}
\renewcommand{\Im}{\mathop{\rm Im}}

\newtheorem{thm}{Theorem}[section]
\newtheorem{proposition}[thm]{Proposition}
\newtheorem{corollary}[thm]{Corollary}
\newtheorem{definition}[thm]{Definition}
\newtheorem{theorem}[thm]{Theorem}
\newtheorem{lemma}[thm]{Lemma}
\newtheorem{remark}[thm]{Remark}

\begin{document}
\title[The porous medium equation on manifolds with conical singularities]%
{Existence and maximal $L^{p}$-regularity of solutions for the porous medium equation on manifolds with conical singularities}

\author{Nikolaos Roidos}
\author{Elmar Schrohe}
\address{Institut f\"ur Analysis, Leibniz Universit\"at Hannover, Welfengarten 1, 30167 Hannover, Germany}
\email{roidos@math.uni-hannover.de, schrohe@math.uni-hannover.de}

\begin{abstract}
We consider the porous medium equation on manifolds with conical singularities and show existence, uniqueness and maximal $L^{p}$-regularity of a short time solution. In particular, we obtain information on the short time asymptotics of the solution near the conical point. Our method is based on bounded imaginary powers results for cone differential operators on Mellin-Sobolev spaces and $R$-sectoriality perturbation techniques.
\end{abstract}

\date{\today}

\maketitle

\section{Introduction}

The porous medium equation is the parabolic diffusion equation 
\begin{eqnarray}\label{e1}
u'(t)-\Delta(u^{m}(t))&=&f(u,t),\quad t\in(0,T_0],\\\label{e2}
u(0)&=&u_{0}.
\end{eqnarray}
It describes the flow of a gas in a porous medium, where $u$ is the density distribution.
We assume here that $m>0$ and $f=f(\lambda,t)$ is a holomorphic function of $\lambda$ on a neighborhood of $\mathrm{Ran}(u_0)$ 
with values in Lipschitz functions in $t$ on $[0,T_0]$. The porous medium equation can be regarded as a nonlinear version of the heat equation, which arises for $m=1$. In general it can model heat transfer, fluid flow or diffusion. 
While traditionally treated on domains in $\R^n$, the porous medium equation has also been considered on closed compact and noncompact manifolds; we refer for this to Bonforte, Grillo \cite{BG}, Huang, Huang, Li \cite{Hua}, Otto \cite{Ott}, Zhang \cite{Zh} and Zhu \cite{Zhu}. 
For related work see also Shao \cite{Shao} for general singular spaces in the spirit of H. Amann \cite{Am1} and Mazzeo, Rubinstein, Sesum \cite{MRS} for the Ricci flow on asymptotically conical surfaces. 

In this article, we study the equation on a manifold with conical singularities. 
The principal application we have in mind is that of a two-dimensional surface with conical singularities, on which we imagine the medium as a film. However, both our methods and results extend to higher dimensions. 
We model the manifold with conical singularities by an $(n+1)$-dimensional compact manifold 
$\mathbb{B}$ with boundary, $n\ge1$, endowed with a degenerate Riemannian metric $g$, which, in a collar neighborhood $[0,1)\times\partial\mathbb{B}$ of the boundary $\partial\mathbb{B}$, admits the warped product structure
\begin{gather*}
g(x,y)=dx^{2}+x^{2}h(y).
\end{gather*}
Here $h$ is a (non-degenerate) Riemannian metric on $\partial\mathbb{B}$, and $(x,y)\in[0,1)\times\partial\mathbb{B}$. 
By $\gD$ we denote the Laplacian associated with this metric. It naturally acts on scales of Mellin-Sobolev spaces $\mathcal{H}^{s,\gamma}_p(\mathbb{B})$, which will be introduced in Section \ref{MS}. 
At this point it suffices to say that $s\in \R$ gives the local Sobolev regularity with respect to an $L^p$ base space, $1<p<\infty$, and $\gg\in \R$ is a weight at the tip.

We will prove existence and maximal regularity of a short time solution to the porous medium equation.
Our main tool will be maximal regularity theory for quasilinear parabolic problems, based on the properties of the associated linearized problem and a theorem of 
Cl\'ement and Li.
We start by considering the Laplacian as an unbounded operator and choosing an appropriate closed extension. The possible domains consist of a Mellin-Sobolev space plus a finite dimensional space of functions with specific asymptotics near the tip, the so-called {\em asymptotics spaces}; see Section 3 for details. 
The maximal regularity approach then allows us to examine the effect of the singularity on the solution of the problem. 

Finding a maximal regularity extension of the Laplacian is not a trivial task. Certain examples of such extensions on Mellin-Sobolev spaces $\cH^{s,\gg}_p(\B)$ of non-negative order (i.e. $s\ge0$) are known explicitly from \cite{RS0}, \cite{RS}, \cite{Sh}. These will be the starting point of our analysis. 
We will, however, have to go beyond the methods used in these articles: 
As a first result, we show that there exist extensions $\underline{\gD}_s$ of the Laplacian 
of maximal regularity also in Mellin-Sobolev spaces of negative order.
In fact, for each $s\in \R$ and suitable $\gg$, depending on $\dim (\B)$, we may choose 
$\cD(\underline{\gD}_s) = \cH^{s+2,\gg+2}_p(\B)\oplus \C$ as a domain in $\cH^{s,\gg}_p(\B)$ in which 
the Laplacian has maximal regularity. 
Here $\C$ stands for the constant functions. 

The second important step, in order to apply the theorem of Cl\'ement and Li, is to control the `trace space', i.e. the interpolation space specifying the admissible initial values. This is connected to the problem of understanding the domains of the complex powers for cone differential operators. 
For a conic manifold, results on the interpolation between Mellin-Sobolev spaces and Mellin-Sobolev spaces plus asymptotics spaces were not known. We manage to obtain suitable embeddings for these spaces in our situation. 

The main difficulty then is to show maximal regularity for the linearized term of the problem, which is a conically degenerate differential operator of second order. 
The idea is to first establish uniform $R$-sectoriality for the operators with frozen coefficients and then to apply perturbation arguments in order to construct $R$-sectorial local approximations of the resolvent. This allows us to show the existence and $R$-sectoriality of the resolvent via a Neumann series expression. 
Our first result is the following.

\begin{theorem}\label{1.1}
Let $\lambda_{1}$ be the largest nonzero eigenvalue of the boundary Laplacian $\Delta_{\partial}$, induced by the metric $h$ on $\partial\mathbb{B}$.
Write
\begin{eqnarray}\label{epsilon}
\overline\gve= -\frac{\mathrm{dim}(\mathbb{B})-2}{2}+
\sqrt{\left(\frac{\mathrm{dim}(\mathbb{B})-2}{2}\right)^{2}-\lambda_{1}}\ >0.
\end{eqnarray}
Let $\gg$ satisfy 
\begin{eqnarray*}
\frac{\dim(\B)-4}2<\gg<\min\left\{\frac{\dim(\B)-4}2+\overline\gve ,\frac{\dim(\B)}2\right\},
\end{eqnarray*}
and choose $p,q$ so large that 
\begin{eqnarray}\label{pq}
\frac{\dim(\B)}p+\frac2q<1 \quad\text{and}\quad \gamma>\frac{\dim(\B)-4}{2}+\frac{4}{q}.
\end{eqnarray}
We consider the extension 
\begin{eqnarray}\label{1.1.1}
\underline\Delta_s: \cD(\underline\gD_s)= \cH^{s+2,\gg+2}_p(\B) \oplus \C\to \cH^{s,\gg}_p(\B). 
\end{eqnarray}
Then, for any $s>-1+\frac{\dim(\B)}{p}+\frac{2}{q}$ and for any strictly positive initial value 
$$
u_{0}\in(\mathcal{H}_{p}^{s+2,\gamma+2}(\mathbb{B})\oplus\mathbb{C},\mathcal{H}_{p}^{s,\gamma}(\mathbb{B}))_{\frac{1}{q},q}
$$
there exists a $T>0$ such that the porous medium equation \eqref{e1}, \eqref{e2},
considered in the space 
$L^{q}(0,T;\mathcal{H}_{p}^{s,\gamma}(\mathbb{B}))$,
has a unique solution 
$$
u\in L^{q}(0,T;\mathcal{H}_{p}^{s+2,\gamma+2}(\mathbb{B})\oplus\mathbb{C})\cap W^{1,q}(0,T;\mathcal{H}_{p}^{s,\gamma}(\mathbb{B})).
$$
\end{theorem} 
According to Theorem III.4.10.2 in \cite{Am} the solution then automatically belongs to 
$$ C([0,T];(\mathcal{H}_{p}^{s+2,\gamma+2}(\mathbb{B})\oplus\mathbb{C},\mathcal{H}_{p}^{s,\gamma}(\mathbb{B}))_{\frac{1}{q},q}).$$
Moreover, we have the following embedding, see Lemma \ref{inter}: For each $\gve>0$
$$
\mathcal{H}_{p}^{s+2-\frac{2}{q}+\varepsilon,\gamma+2-\frac{2}{q}+\varepsilon}(\mathbb{B})\oplus\mathbb{C}\hookrightarrow
(\mathcal{H}_{p}^{s+2,\gamma+2}(\mathbb{B})\oplus\mathbb{C},\mathcal{H}_{p}^{s,\gamma}(\mathbb{B}))_{\frac{1}{q},q})
\hookrightarrow \mathcal{H}_{p}^{s+2-\frac{2}{q}-\varepsilon,\gamma+2-\frac{2}{q}-\varepsilon}(\mathbb{B})\oplus\mathbb{C}.
$$
The assumptions imply that the domain in \eqref{1.1.1} is a subset of $C^\gs$, where $\gs$ has to be slightly larger than $1$. There is not much point here in striving for lower regularity
as the specific character of the conical singularity would be lost.

Finally, we analyze the asymptotics of the solution near the tip of the cone as $t\to 0^+$. 
To this end, we will work with a different extension, namely 
\begin{eqnarray}\label{ext1}
\underline\gD: \cD(\underline\gD):= \cD(\underline\gD_s^2)\to \cD(\underline\gD_s), 
\end{eqnarray}
for $\underline\gD_s$ as above.
The point is that we know $\cD(\underline\gD_s^2)$ rather well in terms of asymptotics and regularity. 
In fact, the asymptotics terms can be computed explicitly from the spectrum of the boundary Laplacian $\gD_\partial$. 
In order to treat the nonlinearity we will have to make the assumption that the first nonzero eigenvalue of $\gD_\partial$ is $<-9/4$ for $\dim(\B)=2$ and $<1-\dim (\B)$ for $\dim (\B)\ge 4$. 
This implies that the parameter $\bar\gve$ defined in Theorem \ref{1.1} is $> 3/2$ for 
two-dimensional $\B$ and $>1$ when the dimension is larger than three.
We let $p,q,s$ as before and assume that 
\begin{eqnarray}\label{ubgamma}
\begin{cases}
-1/2 < \gg< \min\{\overline\gve-2-\frac1{q-1},0\}, & \dim(\B)=2\\
\frac{\dim(\B)-4}2<\gg< \min\left\{\frac{\dim(\B)-4}2+\overline\gve-1
-\frac1{q-1},\frac{\dim(\B)-2}2\right\},
&\dim(\B)\ge4\end{cases}.
\end{eqnarray}
As a consequence, the first non-constant asymptotics term in $\cD(\underline\gD_s^2)$ will be $o(x)$ near the tip $x=0$ (even $o(x^{3/2})$ in the two-dimensional case). It is this fact that eventually will allow us to treat the nonlinearity in the equation. 
This approach breaks down in dimension three.

We obtain the following theorem. 

\begin{theorem}\label{1.2} Let $\dim (\B)\not=3$ and assume that the first non-zero eigenvalue $\gl_1$ of $\gD_\partial$ satisfies $\gl_1<-9/4$ for $\dim (\B)=2$ and $\gl_1<1-\dim(\B)$ for $\dim(\B)\ge4$. 
Choose $s,p,q$ as in Theorem $\ref{1.1}$ and $\gamma$ as in Equation \eqref{ubgamma}. Let $\underline{\Delta}$ be the closed extension of the Laplacian defined in \eqref{ext1}.
Then, for each strictly positive initial value 
$$
u_{0}\in(\mathcal{D}(\underline{\Delta}_s^{2}),\cD(\underline\gD_s))_{\frac{1}{q},q},
$$
there exists a $T>0$ such that the porous medium equation \eqref{e1}, \eqref{e2} has a unique solution 
$$
u\in L^{q}(0,T;\mathcal{D}(\underline{\Delta}_{s}^{2}))\cap 
W^{1,q}(0,T;\cD(\underline \gD_s)).
$$
\end{theorem}
This solution automatically belongs to $C([0,T];(\mathcal{D}(\underline{\Delta}_{s}^{2}),\cD(\underline \gD_s))_{\frac{1}{q},q}$:
The precise description of $\cD(\underline\gD_s^2)$ in Section \ref{bilaplacian}, below, then gives a clear picture of the possible asymptotics.
 
We will start in Section 2 by recalling basic notions and results in connection with maximal regularity. 
In Section 3 we will review Mellin-Sobolev spaces, conically degenerate 
differential operators and their closed extensions. 
The above mentioned result on the existence of bounded imaginary powers for suitable extensions of the Laplacian on Mellin-Sobolev spaces will be established in Section 4. 
Section 5 contains an embedding result for the domains of the complex powers of the Laplacian. The subsequent section is devoted to the study of the porous medium equation in $\cH^{s,\gamma}_p(\B)$ and the proof of Theorem \ref{1.1}. 
The principal result here is 
Theorem \ref{t21} on the maximal regularity of the operators $c-v\Delta$ for $c>0$ sufficiently large and $v$ in the interpolation space $X_{1/q,q}$. Finally, in Section 7 we study the extension \eqref{ext1} in spaces with asymptotics and obtain Theorem \ref{1.2}. 

\subsection*{Acknowledgment} We thank Helmut Abels, Joachim Escher, Matthias Gei{\ss}ert and Christoph Walker for helpful discussions. 

\section{Preliminary Results on Parabolic Problems}\setcounter{equation}{0}
 
In this section let $X_1\overset{d}{\hookrightarrow} X_0$ be a densely injected Banach couple.

\begin{definition}\label{sectorial}
For $\theta\in[0,\pi[$ denote by $\mathcal{P}(\theta)$ the class of all closed densely defined linear operators $A$ in $X_0$ such that 
$$
S_{\theta}=\{z\in\mathbb{C}\,|\, |\arg z|\leq\theta\}\cup\{0\}\subset\rho{(-A)} \quad \mbox{and} \quad (1+|z|)\|(A+z)^{-1}\|\leq K_{\theta}, \quad z\in S_{\theta},
$$
for some $K_{\theta}\geq1$. The elements in $\mathcal{P}(\theta)$ are called {\em sectorial operators of angle $\theta$}.
\end{definition}

Given $A\in \cP(\gt)$ we can define the complex powers $A^z$ for $\Re z<0$ by a Dunford integral; 
composition with $A^k$, $k\in \N$, then yields arbitrary complex powers, cf. Amann \cite[III.4.6.5]{Am}. 
Of particular interest are the purely imaginary powers. 

\begin{definition}
Let $A\in\mathcal{P}(\theta)$, $\theta\in[0,\pi[$. We say that $A$ has {\em bounded imaginary powers} if there exists some $\varepsilon>0$ and $K\geq1$ such that 
$$
A^{it}\in\mathcal{L}(X_0) \quad \mbox{and} \quad \|A^{it}\|\leq K \quad \mbox{for all} \quad t\in[-\varepsilon,\varepsilon].
$$
In this case, there exists a $\phi\geq0$, the so-called power angle, 
such that $\|A^{it}\|\leq M e^{\phi|t|}$ for all $t\in\mathbb{R}$ with some $M\geq1$, and we write $A\in\mathcal{BIP}(\phi)$.
\end{definition} 

The following statement is \cite[Lemma 2.3]{Ro2}: 
\begin{lemma}\label{Nikos}
Let $A\in \cP(\gt)$ for some $\gt>0$ and $x\in \cD(A^\phi)$ for some $0<\phi<\gt$. 
Then, for any $0<\gt'<\gt$ and $0\le\eta<\phi$, 
$$z\mapsto z^\eta A(A+z)^{-1}x \text{ is bounded in } S_{\gt'}.$$ 
\end{lemma}

The proof actually shows slightly more, namely: 
\begin{corollary}\label{Nikos2}
$z\mapsto z^\eta A(A+z)^{-1}\text{ is a uniformly bounded family in } \cL(\cD(A^\phi), X_0).$
\end{corollary}

In UMD spaces (unconditionality of martingale differences property), the notion of the $R$-sectoriality, which is a boundedness condition stronger than the standard sectoriality, characterizes the property of maximal $L^{p}$-regularity for the linear problem. 

\begin{definition}
Let $\theta\in[0,\pi[$. An operator $A\in\mathcal{P}(\theta)$ is called {\em $R$-sectorial of angle $\theta$}, if for any choice of $\lambda_{1},...,\lambda_{N}\in S_{\theta}$, $x_{1},...,x_{N}\in X_0$, $N\in\mathbb{N}$, we have 
\begin{eqnarray}\label{RS1}
\big\|\sum_{\rho=1}^{N}\epsilon_{\rho}\lambda_{\rho}(A+\lambda_{\rho})^{-1}x_{\rho}\big\|_{L^{2}(0,1;X_0)} \leq C \big\|\sum_{\rho=1}^{N}\epsilon_{\rho}x_{\rho}\big\|_{L^{2}(0,1;X_0)},
\end{eqnarray}
for some constant $C\geq 1$, called the {\em $R$-bound}, and the sequence $\{\epsilon_{\rho}\}_{\rho=1}^{\infty}$ of the Rademacher functions. Alternatively, we might have asked that, for some $C'\ge1$, 
\begin{eqnarray}\label{RS2}
\big\|\sum_{\rho=1}^{N}\epsilon_{\rho}A(A+\lambda_{\rho})^{-1}x_{\rho}\big\|_{L^{2}(0,1;X_0)} \leq C' \big\|\sum_{\rho=1}^{N}\epsilon_{\rho}x_{\rho}\big\|_{L^{2}(0,1;X_0)}.
\end{eqnarray}
\end{definition} 

For later use, we state the following elementary observation.

\begin{lemma}\label{easy}
Let $A: X_{1}\rightarrow X_0$ be $R$-sectorial of angle $\theta$, and let $C$ be the $R$-bound. Denote 
$$
S(\theta)=\bigg\{\begin{array}{lll} \sin(\theta) &,& \theta\in[\frac{\pi}{2},\pi)\\ 1&,& \theta\in[0,\frac{\pi}{2})\end{array}.
$$ 
Then, for any $c>0$, $A+c$ is again $R$-sectorial of angle $\theta$ with $R$-bound $\leq C(1+\frac{2}{S(\theta)})$.
\end{lemma}
\begin{proof}
For any $\lambda_{1},...,\lambda_{N}\in S_{\theta}$ and $x_{1},...,x_{N}\in X_0$, $N\in\mathbb{N}$, we have that
\begin{eqnarray*}
\lefteqn{\Big\|\sum_{\rho=1}^{N}\epsilon_{\rho}\lambda_{\rho}(A+c+\lambda_{\rho})^{-1}x_{\rho}\Big\|_{L^{2}(0,1;X_0)}}\\
&\leq&\Big\|\sum_{\rho=1}^{N}\epsilon_{\rho}(c+\lambda_{\rho})(A+c+\lambda_{\rho})^{-1}x_{\rho}\Big\|_{L^{2}(0,1;X_0)}+\Big\|\sum_{\rho=1}^{N}\epsilon_{\rho}c(A+c+\lambda_{\rho})^{-1}x_{\rho}\Big\|_{L^{2}(0,1;X_0)}\\
&\leq&C\Big\|\sum_{\rho=1}^{N}\epsilon_{\rho}x_{\rho}\Big\|_{L^{2}(0,1;X_0)}+C\Big\|\sum_{\rho=1}^{N}\epsilon_{\rho}\frac{c}{c+\lambda_{\rho}}x_{\rho}\Big\|_{L^{2}(0,1;X_0)}\\
&\leq&C\Big(1+2\sup_{\lambda\in S_{\theta}}\Big\{\frac{c}{|c+\lambda|}\Big\}\Big)\Big\|\sum_{\rho=1}^{N}\epsilon_{\rho}x_{\rho}\Big\|_{L^{2}(0,1;X_0)},
\end{eqnarray*}
where Kahane's contraction principle has been applied in the last step, see \cite[Proposition 2.5]{KL1}. 
\end{proof}

Let $A$ be a closed densely defined linear operator $A:\mathcal{D}(A)=X_{1}\rightarrow X_{0}$. Assume that $-A$ generates a bounded analytic semigroup on $X_{0}$. 
This is equivalent to the fact that $c+A\in\mathcal{P}(\theta)$ for some $c\in \C$ and some $\theta>\pi/2$. Consider the abstract first order Cauchy problem
\begin{gather}\label{AP}
\Big\{\begin{array}{lclc} u'(t)+Au(t)=g(t), & t\in(0,T)\\
u(0)=u_{0}&
\end{array} ,
\end{gather}
in the $X_{0}$-valued $L^{q}$-space $L^{q}(0,T;X_{0})$, where $1<q<\infty$, $T>0$.
We say that $A$ has {\em maximal $L^{q}$-regularity}, if for some $q$ (and hence, by a result of Dore \cite{Do}, for all) the unique solution of \eqref{AP} belongs to $L^{q}(0,T;X_{1})\cap W^{1,q}(0,T;X_{0})\cap C([0,T];X_{\frac1q,q})$ and depends continuously on the data $g\in L^{q}(0,T;X_0)$ and $u_{0}$ in the real interpolation space $X_{\frac{1}{q},q}:=(X_{1},X_{0})_{\frac{1}{q},q}$. 

If the space $X_{0}$ is UMD then the following result holds.

\begin{theorem}{\rm (Weis, \cite[Theorem 4.2]{W})}
In a UMD Banach space any $R$-sectorial operator of angle greater than $\frac{\pi}{2}$ has maximal $L^{q}$-regularity. 
\end{theorem}

For an alternative approach to maximal $L^{q}$-regularity based on a Hardy-Littlewood majorant type property of the resolvent of an operator instead of Rademacher boundedness, or for sectorial operators admitting a special type of operator-valued bounded $H^{\infty}$-calculus, we refer to \cite{Ro2} and \cite{Ro1}, respectively. 

\begin{remark}\rm \label{rem1}
In a UMD space, an operator $A\in \mathcal{BIP}(\phi)$ with $\phi<\pi/2$ is
$R$-sectorial with angle greater than $\pi/2$ by \cite[Theorem 4]{CP} and hence has maximal $L^q$-regularity. This also is a classical result by Dore and Venni \cite{DV}. 
\end{remark}

Next, we consider the abstract quasilinear parabolic problem of the form
\begin{gather}\label{QL}
\Big\{\begin{array}{lclc} u'(t)+A(u(t))u(t)=f(t,u(t))+g(t), & t\in(0,T_{0});\\
u(0)=u_{0}&	
\end{array} 
\end{gather}
in $L^{q}(0,T_{0};X_{0})$, such that $\mathcal{D}(A(u(t)))=X_{1}$, $1<q<\infty$, and $T_{0}$ is finite. Maximal regularity for the solution of the linearized problem together with Lipschitz continuity will also imply existence and maximal regularity for the solution of the original problem, as the following well known theorem guarantees.

\begin{theorem}\label{CL} {\rm (Cl\'ement and Li, \cite[Theorem 2.1]{CL}) }
Assume that there exists an open neighborhood $U$ of $u_0$ in $X_{\frac{1}{q},q}$ such that $A(u_0): X_{1}\rightarrow X_{0}$ has maximal $L^{q}$-regularity and that 
\begin{itemize}
\item[(H1)] $A\in C^{1-}(U, \cL(X_1,X_0))$, 
\item[(H2)] $f\in C^{1-,1-}([0,T_0]\times U, X_0)$,
\item[(H3)] $g\in L^q(0,T_0; X_0)$.
\end{itemize}
Then there exists a $T>0$ and a unique $u\in L^q(0,T;X_1)\cap W^{1,q}(0,T;X_0)\cap C([0,T];X_{\frac{1}{q},q})$ solving the Equation \eqref{QL} on $(0,T)$.
\end{theorem}

\section{Cone Differential Operators and Mellin-Sobolev Spaces}\setcounter{equation}{0}

\subsection{Mellin-Sobolev spaces}\label{MS}
For $s\in \N_0$, $\cH^{s,\gg}_p(\B)$ is the space of all functions $u$ in $H^s_{p,loc}(\B^\circ)$
such that, near the boundary, 
\begin{eqnarray}\label{measure}
x^{\frac{n+1}2-\gamma}(x\partial_x)^j\partial_y^{\alpha}(\omega(x) u(x,y))
\in L^p\Big([0,1]\times \partial \B, \frac{dx}xdy\Big),\quad j+|\ga|\le s,
\end{eqnarray}
where $\omega$ is a cut-off function, i.e. a smooth non-negative
function $\go$ on $\R$ with $\go(x)=1$ for $x$ near $x=0$ and $\go(x)=0$ for $x\ge1$. 

For an arbitrary $s\in\R$, define the map 
$$\cS_\gamma: C^\infty_c(\R^{n+1}_+) \to C^\infty_c(\R^{n+1}),\qquad 
v(t,y) \mapsto e^{(\gg-\frac{n+1}2)t} v(e^{-t},y).$$
Let moreover $\kappa_{j}:U_{j}\subseteq\partial \B\to\R^n$, $j=1,\ldots,N,$
be a covering of $\partial \B$ by coordinate charts and $\{\varphi_{j}\}$ 
a subordinate partition of unity. 

\begin{definition}\label{dms}
$\cH^{s,\gg}_p(\B)$, $s,\gg\in\R$, $1<p<\infty$, is the space of all distributions on $\B^\circ$ such that 
\begin{equation}\label{norm}
\begin{array}{lcl} \|u\|_{\cH^{s,\gg}_p(\B)} \! & \! = \! & \! \displaystyle \sum_{j=1}^{N}\|\cS_{\gamma}(1\otimes\kappa_{j})_{*} (\omega\varphi_{j}u)\|_{H^{s}_{p}(\R^{1+n})}+\| (1-\omega)u)\|_{H^{s}_{p}(\B)} \end{array} 
 \end{equation}
is defined and finite. 
\end{definition}

Here, $\omega$ is a (fixed) cut-off 
function and $*$ refers to the push-forward of distributions. 
Up to equivalence of norms, this construction is independent of the choice of 
$\go$ and the $\kappa_j$. See~\cite[Section 2.2]{CSS1}.
Clearly, all the spaces $\cH^{s,\gg}_p(\B)$ are UMD spaces. 

We recall the following property for the Mellin-Sobolev spaces, which extends the standard Sobolev embedding theorem. 

\begin{lemma}\label{c0}Let $1\le p<\infty$ and $s>(n+1)/p$. 
Then $\|uv\|_{\cH^{s,\gg}_p(\B)}\le c \|u\|_{\cH^{s,\gg}_p(\B)}\|v\|_{\cH^{s,(n+1)/2}_p(\B)}$ for suitable $c>0$. 
In particular, $\cH^{s,\gg}_p(\B)$ is a Banach algebra
\footnote{Up to the choice of an equivalent norm. However, we will not distinguish the norms in the sequel.}
, whenever $s>(n+1)/p$ and $\gg\ge (n+1)/2$. Moreover, if $s>(n+1)/p$, a function $u$ in 
$\cH^{s,\gg}_p(\B)$ is continuous on $\B^\circ$, and, near $\partial\B$, 
\begin{eqnarray*}
|u(x,y)|\le c' x^{\gg-(n+1)/2} \|u\|_{\cH^{s,\gg}_p(\B)}
\end{eqnarray*}
for a constant $c'>0$.
\end{lemma}
\begin{proof}
This is Corollary 2.8 and Corollary 2.9 in \cite{RS}.
\end{proof}

The following is a slight improvement of Corollary 2.10 in \cite{RS}. We include the proof for the convenience of the reader. 

\begin{corollary}\label{multiplier}
Let $1<p,q<\infty$, $\gs\geq0$, $\gamma\in\mathbb{R}$. 
Then multiplication by an element $m$ in $\cH^{\gs+\frac{n+1}{q},\frac{n+1}{2}}_q(\B)$ defines a bounded map on $\cH^{s,\gg}_p(\B)$ for each $s\in(-\gs,\gs)$.
\end{corollary}
\Proof {Let $v\in \cH^{s,\gg}_p(\B)$ and denote by $\sim$ equivalence of norms. We can assume that $v$ is supported near the boundary in a single coordinate neighborhood. Then
\begin{eqnarray*}
\lefteqn{\|mv\|_{\cH^{s,\gg}_p(\B)} 
\sim \|e^{(\gg-(n+1)/2)t}m(e^{-t},y)v(e^{-t},y)\|_{H^s_p(\R^{n+1})}}\\
&\le& c_1 \|m(e^{-t},y)\|_{C^{\gs}_*(\R^{n+1})}
	\|e^{(\gg-(n+1)/2)t}v(e^{-t},y)\|_{H^s_p(\R^{n+1})}\\
&\le& c_2\|m(e^{-t},y)\|_{H^{\gs+(n+1)/q}_q(\R^{n+1})}
	\|e^{(\gg-(n+1)/2)t}v(e^{-t},y)\|_{H^s_p(\R^{n+1})}\\
&\sim&\|m\|_{\cH^{\gs+(n+1)/q,(n+1)/2}_q(\B)} \|v\|_{\cH^{s,\gg}_p(\B)}.
\end{eqnarray*} 
Here, the first inequality is due to the fact that multiplication by functions in the Zygmund space
$C^\gs_*$ defines a bounded operator in $H^s_p$ for $-\gs<s<\gs$, cf.\ \cite[Section 13, Proposition 9.10]{Tay}, 
and the second is a consequence of the fact that $H^{\gs+(n+1)/q}_q(\R^{n+1})
\hookrightarrow C^{\gs}_*$, cf.\ \cite[Section 13, Proposition 8.5]{Tay}.}

\begin{theorem}\label{dual} For $s,\gg\in \R$ and $1<p<\infty$, the scalar product in $\cH^{0,0}_2(\B)$ 
yields an identification of the dual space to $\cH^{s,\gg}_p(\B)$ with $\cH^{-s,-\gg}_{p'}(\B)$, 
where $p'$ is conjugate to $p$, i.e $1/p+1/p'=1$. 
\end{theorem}
\begin{proof}
Follows by Definition \ref{dms} and the analogous result for the standard Sobolev spaces.
\end{proof}

The following result is immediate from \cite[Lemma 5.4]{CSS1}. 
\begin{lemma}\label{int}
Let $s_0,s_1,\gg_0,\gg_1\in \R$, $0\le\gt<1$ and $1<p,q<\infty$. Then, for arbitrary $\gve>0$,
$$(\cH_p^{s_0,\gg_0}(\B),\cH_p^{s_1,\gg_1}(\B))_{\gt,q}\hookrightarrow \cH_p^{s-\gve,\gg-\gve}(\B)
$$
with $s=(1-\gt)s_0+\gt s_1$ and $\gg=(1-\gt)\gg_0+\gt \gg_1$.
\end{lemma} 

Conversely we have, with the same notation and a very similar proof:
\begin{lemma}\label{int2}For arbitrary $\gd,\gve>0$, 
$$\cH_p^{s+\gd,\gg+\gve}(\B)\hookrightarrow (\cH_p^{s_0,\gg_0}(\B),\cH_p^{s_1,\gg_1}(\B))_{\gt,q}.
$$
\end{lemma} 

In the case of complex interpolation between Mellin-Sobolev spaces of the same weight, we have a sharp result on the order of the resulting Mellin-Sobolev space as follows.

\begin{lemma}\label{sharpint}
Let $s_0,s_1,\gg\in \R$ and $0<\gt<1$. Then, 
$$
[\cH_p^{s_0,\gg}(\B),\cH_p^{s_1,\gg}(\B)]_{\gt}= \cH_p^{s,\gg}(\B), 
$$
with $s=(1-\gt)s_0+\gt s_1$.
\end{lemma}
\begin{proof}
Recall that the statement is true for the standard Sobolev spaces in $\mathbb{R}^{n}$ (see e.g. Equation 4.(2.18) in \cite{Tay1}). Let $\go\equiv 1 $ on $[0,1/2)$ and let $(V_i,\tau_i)_i$ 
with $\tau_{i}:V_{i}\subset\B\to\R^n$, $i=1,\ldots,M,$ be a covering of $\B\setminus\{[0,1/3)\times\partial\B\}$ by coordinate charts and $\{\psi_{i}\}$ a subordinate partition of unity. With the notation as in Definition \ref{dms}, we have the following equivalence
\begin{eqnarray*}
\lefteqn{u\in \cH_p^{s,\gg}(\B)}\\
&\Longleftrightarrow& \cS_{\gamma}(1\otimes\kappa_{j})_{\ast}(\omega\varphi_{j}u), \,\,\, ( \tau_{i})_{\ast}(\psi_{i}(1-\omega)u)\in H_{p}^{s}(\mathbb{R}^{n}), \,\,\, \forall\, i,j\\
&\Longleftrightarrow& \cS_{\gamma}(1\otimes\kappa_{j})_{\ast}(\omega\varphi_{j}u), \,\,\, ( \tau_{i})_{\ast}(\psi_{i}(1-\omega)u)\in [H_p^{s_0,\gg}(\mathbb{R}^{n}),H_p^{s_1,\gg}(\mathbb{R}^{n})]_{\gt}, \,\,\, \forall\, i,j\\
&\Longleftrightarrow& \omega\varphi_{j}u, \,\,\, \psi_{i}(1-\omega)u\in [\cH_p^{s_0,\gg}(\B),\cH_p^{s_1,\gg}(\B)]_{\gt}, \,\,\, \forall\, i,j \\
&\Longleftrightarrow& \omega u, \,\,\, (1-\omega)u\in [\cH_p^{s_0,\gg}(\B),\cH_p^{s_1,\gg}(\B)]_{\gt}\\
&\Longleftrightarrow& u\in [\cH_p^{s_0,\gg}(\B),\cH_p^{s_1,\gg}(\B)]_{\gt},
\end{eqnarray*}
where we have used the linearity of the push-forward together with Proposition 4.2.1 in \cite{Tay1}. 
\end{proof}

\subsection{Cone differential operators}
A cone differential operator or conically degenerate operator $A$ of order $\mu$ is a differential operator of order $\mu$ on the interior $\B^\circ$ of $\B$ which, in local coordinates near the boundary, can be written in the form 
\begin{eqnarray}\label{conediffop}
A=x^{-\mu}\sum_{j=0}^\mu a_j(x) (-x\partial_x)^j\quad \mbox{with} \quad a_j\in C^{\infty}([0,1),\mathrm{Diff}^{\mu-j}(\partial\mathbb{B})).
\end{eqnarray} 
The operator $A$ is called $\B$-elliptic (or degenerate elliptic), if the principal pseudodifferential symbol $\gs^\mu_\psi(A)$ is invertible on $T^*\B^\circ$ and, moreover, in local coordinates 
$(x,y)$ near the boundary and corresponding covariables 
$(\xi,\eta)$, the rescaled principal symbol 
\begin{eqnarray}\label{rps}
x^{\mu}\gs^\mu_\psi(A)(x,y,\xi/x,\eta) 
= \sum_{j=0}^\mu\gs^{\mu-j}_\psi (a_j(x))(y,\eta)(-i\xi)^j
\end{eqnarray}
is invertible up to $x=0$. 

Moreover, one associates to a cone differential operator its conormal symbol. This is the operator 
family $\gs_M(A):\C\to \cL(H^s_p(\partial \B),H^{s-\mu}_p(\partial \B))$ defined by
\begin{eqnarray}\label{conormal}
\gs_M(A)(z) =\sum_{j=0}^\mu a_j(0) z^j: H^s_p(\partial \B)\to H^{s-\mu}_p(\partial \B).
\end{eqnarray}
One is mostly interested in the points where $\gs_M(A)$ is not invertible. 
For this, the precise choice of $s$ and $p$ in \eqref{conormal} is irrelevant due to the spectral invariance of differential operators in these spaces; it is often convenient to work with $s=0$ and $p=2$. 

The Laplacian $\Delta$ induced by the metric $g$ is a second order cone differential operator. 
Near the boundary $\gD$ can be written in the form
\begin{gather}\label{lap}
\Delta=\frac{1}{x^2}\big((x\partial_{x})^{2}+(n-1)(x\partial_{x})+\Delta_{\partial}\big),
\end{gather}
where $\Delta_{\partial}$ is the boundary Laplacian induced by $h$. 
Hence $\gD$ is $\B$-elliptic. 
The conormal symbol $\gs_M(\Delta)$ 
of $\Delta$ 
is given by
\begin{eqnarray}\label{conormalDelta}
\sigma_M(\Delta)(z)=z^2-(n-1)z+\Delta_\partial.
\end{eqnarray}

\subsection{Closed extensions}
A conically degenerate operator $A$ acts in a natural way on the scales of weighted Mellin-Sobolev spaces: 
$$A: \cH^{s+\mu,\gg+\mu}_p(\B)\to \cH^{s,\gg}_p(\B)$$
is bounded for all $s,\gg\in\R$, $1<p<\infty$.

Sometimes it is necessary to consider a cone differential operator $A$ as an unbounded operator in a fixed space $\cH^{s,\gg}_p(\B)$. 
If $A$ is $\B$-elliptic, then the domain of the minimal extension, i.e. the closure of $A$ considered as an operator on $C^\infty_c(\B^\circ)$, is 
\begin{eqnarray}\label{dmin} 
\cD(A_{s,\min})=\Big\{ u\in \bigcap_{\gve>0}\cH_p^{s+\mu,\gg+\mu-\gve}(\B): 
x^{-\mu}\sum_{j=0}^\mu a_j(0)(-x\partial_x)^ju\in \cH^{s,\gg}_p(\B)\Big\}.
\end{eqnarray}
If, in addition, the conormal symbol of $A$ is invertible for all $z$ with 
$\Re z = (n+1)/2-\gg-\mu$, then
$$\cD(A_{s,\min})= \cH^{s+\mu,\gg+\mu}_p(\B).$$ 

The domain of the maximal extension, defined by
$$\cD(A_{s,\max})=\{u\in \cH^{s,\gg}_p(\B): Au\in \cH^{s,\gg}_p(\B)\},$$
satisfies 
$$
\cD(A_{s,\max})=\cD(A_{s,\min}) \oplus \cE,
$$
where $\cE$ is a finite-dimensional space consisting of linear combinations of functions of the form $x^{-\rho}\log^k x \, c(y)$ with $\rho\in \C$, $k\in \N_0$ and a smooth function $c$ on the cross-section. The space $\cE$ can be chosen independent of $s$. This result has a long history, see e.g. \cite{BrueningSeeley88}, \cite{Le}, \cite{Sh}, \cite{se}; the present version is due to Gil, Krainer and Mendoza \cite{GKM}.

\subsection{Extensions of the Laplacian}
We are interested in the values of $z$, for which $\gs_M(\Delta)$ is not invertible.
We denote by
 $0=\lambda_0>\lambda_1>\ldots$ the distinct eigenvalues of $\Delta_\partial$ and by
 $E_0,\,E_1,\ldots$ the corresponding eigenspaces. 
Moreover, we let $\pi_j\in\cL(L^2(\partial\B))$ be the orthogonal projection 
onto $E_j$.
 
The {\em non}-bijectivity points of $\sigma_M(\Delta)$ are the 
points $z=q_j^+$ and $z=q_j^-$ with 
\begin{equation}\label{pjpm}
q_j^\pm=\mbox{$\frac{n-1}{2}\pm
\sqrt{\big(\frac{n-1}{2}\big)^2-\lambda_j}$},
\qquad j\in\N_0.
\end{equation}
Note the symmetry $q_j^+=(n-1)-q_j^-$. 
 It is straightforward to see that 
\begin{eqnarray}\label{inverse}
(z^2-(n-1)z+\Delta_\partial)^{-1}=
 \sum_{j=0}^\infty\frac{1}{(z-q_j^+)(z-q_j^-)}\pi_j.
\end{eqnarray}

Hence, in case $\text{\rm dim}\,\B\not=2$, where the $q^\pm_j$ are all different, 
the inverse to $\gs_M(\Delta)$ has only simple poles in the points $q^\pm_j$.
For $\text{\rm dim}\,\B=2$ the poles at $q_j^\pm$, $j\not=0$, are simple, 
while there is a double pole at $q_0^+=q_0^-=0$.

With $q_j^\pm$, $j\not=0$, we associate the function spaces 
 $$\cE_{q_j^\pm}=\omega\,x^{-q_j^\pm}\otimes E_j=
 \{\omega(x)\,x^{-q_j^\pm}\,e(y): e\in E_j\},\ j\in\N.
 $$
For $j=0$ we let
\begin{equation}\label{ep0pm}
\cE_{q_0^\pm}=
 \begin{cases}
 \go\otimes E_0+\go\log x\otimes E_0, & 
 \text{\rm dim}\,\B=2\\
 \omega\,x^{q_0^\pm} \otimes E_0,& \text{\rm dim}\,\B\not=2
 \end{cases}.
 \end{equation}
 For later use note that $\gD$ maps the spaces $\cE_{q_j^\pm}$ to 
 $C^\infty_c(\B^\circ)$. 
 
 Furthermore, we introduce the sets $I_\gg$, $\gg\in\R$, by 
 $$I_\gamma=\{q_j^\pm: j\in\N_0\}\cap
 \,\mbox{$]\frac{n+1}{2}-\gamma-2,\frac{n+1}{2}-\gamma[$}.$$

The following is Proposition 5.1 in \cite{Sh} combined with Theorem 3.6 in \cite{GKM}:
 
 \begin{proposition}\label{maxmin}
 The domain of the maximal extension of $\Delta$ in
 $\cH^{s,\gamma}_p(\B)$ is 
 $$\cD(\Delta_{s,\max})=\cD(\Delta_{s,\min})\oplus
 \bigoplus_{q_j^\pm\in I_\gamma}\cE_{q_j^\pm}.$$
 In case $q_j^\pm\not=\frac{n+1}{2}-\gamma-2$ for all $j$, the minimal 
 domain is $\cD(\Delta_{s,\min})=\cH^{s+2,\gamma+2}_p(\B)$. 
 \end{proposition}

\begin{corollary}\label{domdelta}
The domains of the closed extensions of $\Delta$ are the sets of the form 
$\cD(\Delta_{s,\min})\oplus\underline\cE,$
where $\underline\cE$ is any subspace of 
$ \mathop{\oplus}_{q_j^\pm\in I_\gamma}\cE_{q_j^\pm}.$
\end{corollary}

\begin{definition}\label{ext}
Given a subspace $\underline\cE_{q_j^\pm}$ of $\cE_{q_j^\pm},$ 
we define the space $\underline{\cE}_{q_j^\pm}^\perp$ as follows: 
 \begin{itemize}
 \item[{\rm i)}] If either $q_j^\pm\not=0$ or 
$\dim(\B)\not=2$, there exists a unique subspace 
$\underline{E}_j\subseteq E_j$ such that 
$\underline{\cE}_{q_j^\pm}=\omega\, x^{-q_j^\pm}\otimes \underline{E}_j$. 
Then we set 
$$\underline{\cE}_{q_j^\pm}^\perp
=\omega\,x^{-q_j^\mp}\otimes\underline{E}_j^\perp,$$
where $\underline{E}_j^\perp$ is the orthogonal complement of 
$\underline{E}_j$ in $E_j$ with respect to the $L^2(\partial\B)$-scalar product. 
\item[{\rm ii)}] For $\dim(\B)=2$ and $q_0^\pm=0$ define 
$\underline{\cE}_0^\perp=\{0\}$ if $\underline{\cE}_0=\cE_0$, 
$\underline{\cE}_0^\perp=\cE_0$ if $\underline{\cE}_0=\{0\}$, and 
$\underline{\cE}_0^\perp=\underline{\cE}_0$ if 
$\underline{\cE}_0=\go\otimes E_0$. 
\end{itemize}
\end{definition}
 Note that $\underline{\cE}_{q_j^\pm}^\perp$ is a subspace of
 $\cE_{q_j^\mp}$.
 
We will now confine ourselves to closed extensions $\underline{\Delta}_{s}$ of $\gD$ 
in $\cH^{s,\gg}_p(\B)$ with domains 
$$\cD(\underline{\Delta}_{s})=\cD(\Delta_{s,\min}) 
\oplus\mathop{\mbox{$\bigoplus$}}_{q_j^\pm\in I_\gamma}\underline{\cE}_{q_j^\pm}
\subseteq \cH^{s,\gg}_p(\B)$$
chosen according to the following rules: 
\begin{itemize}
 \item[{\rm(i)}] If $q_j^\pm\in I_{\gamma}\cap I_{-\gamma}$, then 
 $\underline{\mathcal{E}}_{q_j^\pm}^\perp=\underline{\mathcal{E}}_{(n-1)-q_j^\pm}$. 
 \item[{\rm(ii)}] If $\gamma\ge0$ and $q_j^\pm\in I_{\gamma}\setminus I_{-\gamma}$,
 then $\underline{\mathcal{E}}_{q_j^\pm}=\mathcal{E}_{q_j^\pm}$. 
 \item[{\rm(iii)}] If $\gamma\le0$ and 
 $q_j^\pm\in I_{\gamma}\setminus I_{-\gamma}$, 
then $\underline{\mathcal{E}}_{q_j^\pm}=\{0\}$. 
\end{itemize}
 In particular, $\cD(\underline{\Delta}_{s})=\cD(\Delta_{s,\max})$ if $\gamma\ge1$ and $\cD(\underline{\Delta}_{s})=\cD(\Delta_{s,\min})$ if $\gamma\le-1$. 
 
\begin{theorem}\label{elldomain2} Let $s\ge0$, $\gt\in[0,\pi[$, and $\phi>0$. For $|\gg|< \dim(\B)/2$ let $\underline{\Delta}_{s}$ be an extension with domain in $\cH^{s,\gg}_p(\B)$ chosen as above. Then $c-\underline \Delta_s\in \cP(\gt)\cap\mathcal{BIP}(\phi)$ for suitably large $c>0$. 
\end{theorem}
\begin{proof}
For $s=0$ this follows from \cite[Theorem 2.9 and Remark 2.10]{RS0} with Theorems 5.7 and 4.3 in \cite{Sh}. For $s>0$ we apply Theorem 3.3 in \cite{RS}. 
\end{proof}

\section{Bounded Imaginary Powers for the Laplacian on Negative Order Mellin-Sobolev Spaces}\setcounter{equation}{0}

Let $X$ be a complex Banach space and $X^{\ast}$ its dual. 
The adjoint $A^{\ast}$ of a sectorial operator $A$ with bounded imaginary powers will also be sectorial on the same sector and have bounded imaginary powers with the same power angle, provided $\cD(A^{\ast}) $ is dense in $X^{\ast}$, see Propositions 1.3(v) and 2.6(v) in \cite{DHP}.
We will use this fact in order to extend Theorem \ref{elldomain2} above to the case of Mellin-Sobolev spaces of negative order.

\begin{theorem}\label{tadj}For $s\geq 0$, $|\gamma|<\frac{\mathrm{dim}(\mathbb{B})}{2}$ and $1<p<\infty$ consider the closed extension $\underline{\Delta}_{-s}$ of the Laplacian 
with domain 
\begin{eqnarray}\label{extneg}
\mathcal{D}(\underline{\Delta}_{-s})=\cD(\underline\gD_{-s,\min})\oplus\bigoplus_{q_{j}^{\pm}\in I_{\gamma}}\underline{\mathcal{E}}_{q_{j}^{\pm}},
\end{eqnarray}
where the asymptotics spaces $\underline{\mathcal{E}}_{q_{j}^{\pm}}$ satisfy the conditions {\rm (i), (ii)} and {\rm(iii)} in Definition $\ref{ext}$.
Then, for any $\theta\in[0,\pi)$ and $\phi>0$, there exists some $c>0$ such that $c-\underline{\Delta}_{-s}\in\mathcal{P}(\theta)\cap\mathcal{BIP}(\phi)$.
\end{theorem}
\begin{proof}
We will show that \eqref{extneg} is the adjoint of a closed extension of the Laplacian on another Mellin-Sobolev space, namely, of the extension $\underline\gD_{s,-\gg,p'}$ with the domain 
\begin{eqnarray}\label{extnegadjoint}
\cD(\underline \gD_{s,-\gg,p'})= \cD(\gD_{s,-\gg, p',\min}) \oplus \bigoplus_{q^\pm_j\in I_{\gg}}\underline\cE_{q_j^\pm}^\perp,
\end{eqnarray}
where $1/p+1/p'=1$. Here $\cD(\gD_{s,-\gg, p',\min})$ denotes the minimal domain of $\gD$ in $\cH^{s,-\gg}_{p'}(\mathbb{B})$. 

In fact, by Theorem 5.3 in \cite{Sh}, the domain of the adjoint of $\underline{\Delta}_{s,-\gg,p'}$ is
$$
\mathcal{D}((\underline{\Delta}_{s,-\gg,p'})^{\ast})=
\cD(\gD_{-s,\gg, p,\min} ) \oplus\bigoplus_{q^\pm_j\in I_{\gg}}(\underline\cE_{q_j^\pm}^\perp)^\perp.$$
This coincides with \eqref{extneg}, since, by the definition of the orthogonal complements, 
$$ (\underline{\mathcal{E}}_{q_{j}^{\pm}}^{\perp})^{\perp}
=\underline{\mathcal{E}}_{q_{j}^{\pm}}.$$ 

In view of the above stated result on the adjoints of sectorial operators with bounded imaginary powers, 
it suffices to show that, given any $\theta\in[0,\pi)$ and any $\phi>0$, the operator $c-\gD_{s,-\gg,p'}$ with domain \eqref{extnegadjoint} is sectorial of angle $\theta$ and has bounded imaginary powers with angle $\phi$, provided $c$ is large enough. 
This in turn will follow directly from Theorem \ref{elldomain2}, provided the asymptotics spaces satisfy the conditions (i), (ii), and (iii) in Definition \ref{ext}. 

In order to see this we note first that $q_j^++q_j^- = n-1$, so that $q_j^\pm \in I_\gg$ if and only if $q_j^\mp \in I_{-\gg}$. As a consequence, we may rewrite the direct sum on the right hand side of 
\eqref{extnegadjoint} as the direct sum over all $q_j^\mp\in I_{-\gg}$. 
It remains to check that the asymptotics spaces have the properties in (i), (ii) and (iii). 
We distinguish the three cases
\begin{itemize}
\item 
We have $q_j^\mp\in I_\gg\cap I_{-\gg}$ if and only if $q_j^\pm\in I_\gg\cap I_{-\gg}$.
Hence the assumption 
$\underline\cE_{q_j^\pm}^\perp = \underline\cE_{q_j^\mp}$ for $\gD_{-s}$ shows that also 
$\underline\cE_{q_j^\mp}^\perp = \underline\cE_{q_j^\pm}$, so that (i) holds for the extension with domain \eqref{extnegadjoint}.

\item If $\gg\ge 0$ and $q_j^\mp\in I_{-\gg}\setminus I_\gg$, 
then $q_j^\pm\in I_\gg\setminus I_{-\gg}$. Since $\underline\gD_{-s}$ 
satisfies condition (ii), $\underline\cE_{q_j^\pm}=\cE_{q_j^\pm} $. Therefore
$\underline\cE_{q_j^\pm}^\perp =\{0\}$, which says that the domain \eqref{extnegadjoint} 
satisfies (iii).

\item Similarly, if $\gg\le 0$ and $q_j^\mp\in I_{-\gg}\setminus I_\gg$, 
then $q_j^\pm\in I_\gg\setminus I_{-\gg}$. As (iii) holds for $\underline\gD_{-s}$, 
we have $\underline\cE_{q_j^\pm}=\{0\} $. Therefore
$\underline\cE_{q_j^\pm}^\perp =\cE_{q_j^\pm}$, and the domain \eqref{extnegadjoint} 
satisfies (ii).
\end{itemize}
\end{proof}

\subsection{Application}
By definition, $\gl_1$ is the largest negative eigenvalue of the Laplacian $\gD_\partial$ induced 
by the metric $h$ on the boundary. We let 
\begin{eqnarray}\label{epsilonbar}
\overline \varepsilon_n =-\frac{n-1}{2}+\sqrt{(\frac{n-1}{2})^{2}-\lambda_{1}} = -q_1^->0.
\end{eqnarray} 
We choose the weight $\gg$ in the interval 
\begin{eqnarray}\label{choicegg}
\frac{n-3}2<\gg<\min\left\{\frac{n-3}2+\overline\gve_n ,\frac{n+1}2\right\}
\end{eqnarray}
and consider the closed extension $\underline{\Delta}_{s}$ of the Laplacian 
in the Mellin-Sobolev space $\mathcal{H}_{p}^{s,\gamma}(\mathbb{B}):= X_{0}$, $1<p<\infty$, $s\in\mathbb{R}$, given by 
\begin{gather}\label{dds}
\mathcal{D}(\underline{\Delta}_{s})=\mathcal{H}_{p}^{s+2,\gamma+2}(\mathbb{B})\oplus\mathbb{C}:=X_1.
\end{gather}
Note that $\frac{n+1}2-\gg-2 =\frac{n-3}2-\gg \in (-\overline\gve_n,0)$ does not coincide with one of the $q_j^\pm$ so that $\cH^{s+2,\gg+2}_p(\B)$ 
is the minimal domain of $\gD$ in $\cH^{s,\gg}_p(\B)$. In addition, the choice of the single asymptotics space $\C$, the constant functions, meets the requirement in Theorem \ref{elldomain2}. 

We can therefore deduce from Theorem \ref{tadj} that for any $\phi>0$, 
the operator $c-\underline\gD_s$ with domain \eqref{dds} and $\gg$ satisfying 
\eqref{choicegg} has $\mathcal{BIP}(\phi)$ provided $c$ is sufficiently large.
But more is true. Let $c\notin \R_{\le0}$. Based on the results in \cite{ScSe} it was shown in 
\cite[Theorem 4.1]{RS} that $c-\underline\Delta_s: \cH^{s+2,\gg+2}_p(\B)\oplus \C$ is invertible for 
$s\ge0$. Since it follows from Corollaries 3.3 and 3.5 in \cite{ScSe} that the kernel and the cokernel are independent of $s$ and $p$, we find that it is also invertible for $s<0$. 

\begin{theorem}\label{t11}
For any $c>0$, $\gt\in [0,\pi)$ and $\phi>0$, the operator $c-\underline\gD_s$ with domain \eqref{dds} and $\gg$ satisfying 
\eqref{choicegg} belongs to $\mathcal P(\gt)\cap \mathcal{BIP}(\phi)$.
\end{theorem}
\begin{proof} 
As pointed out above, we know the result already for large $c$. Since the resolvent
exists outside $\R_{\le 0}$ and $\gt\in [0,\pi)$ is arbitrary, we have $c-\underline\gD_s\in \mathcal P(\gt)$. According to Lemma III.4.7.4 in \cite{Am} we find for any $\phi>0$ a $c'>0$ and $M\ge0$ such that the norm of $(c'-\underline\Delta_s)^{z}$ in $\cL(\cH^{s,\gamma}_p(\B))$ is $\le M e^{\phi|\Im z|}$ for $-1<\Re z<0$. 
Denote by $\Gamma_{\theta}$ the positively oriented boundary of the sector $S_\theta$. Using the Dunford integral for the complex powers, see III.4.6.5 in \cite{Am}, we see that 
\begin{eqnarray*}
\lefteqn{\|(c-\underline\Delta_s)^{z}\|\le \|(c'-\underline\Delta_s)^{z}\| + \|(c-\underline\Delta_s)^{z}-(c'-\underline\Delta_s)^{z}\|}\\
&\le& M e^{\phi|\Im z|} +\frac{|c-c'|}{2\pi}\Big\|\int_{\Gamma_\theta}(-\lambda)^z (c-\underline\Delta_s+\lambda)^{-1}(c'-\underline\Delta_s+\lambda)^{-1}d\lambda\Big\| \\
&\le& M e^{\phi|\Im z|} + M' e^{(\pi-\theta)|\Im z|},
\end{eqnarray*}
for a suitable constant $M'$.
Since $\theta$ can be taken arbitrarily close to $\pi$, we have $c-\underline\Delta_s\in \mathcal{BIP}(\phi)$ by Lemma III.4.7.4 in \cite{Am}.
\end{proof}

\section{Embeddings for the Domain of the Complex Powers of the Laplacian}\setcounter{equation}{0}

At this point, we are able show how the interpolation between a Mellin-Sobolev space on one hand and a direct sum of a Mellin-Sobolev space and the constant functions on the other can be controlled. 
As a consequence, standard theory for sectorial operators furnishes some information about the domain of the complex powers of the Laplacian. We first prove the following.

\begin{lemma}\label{lk2}
Let $s,\gamma\in\mathbb{R}$ and $p\in(1,\infty)$. Suppose $u\in \cH^{s,\gg}_p(\B)\cap H^{s+2}_{p,loc}(\B^\circ)$ and 
that in local coordinates near the boundary, $x\partial_xu$ and $\partial_{y_j}u$ belong to 
$\cH_p^{s+1,\gg+2}(\B)$. 
Then $u\in \mathcal{D}(\Delta_{s,\min})+\mathbb{C}$. 
\end{lemma}
\begin{proof} 
From \eqref{lap} and the assumption we conclude that $\Delta u\in\cH_p^{s,\gg}(\B)$. Therefore, $u$ belongs to the maximal domain of $\Delta$ in $\cH_p^{s,\gg}(\B)$, and
Proposition \ref{maxmin} shows that 
 $$ u\in \mathcal{D}(\Delta_{s,\min})\oplus
 \bigoplus_{q_j^\pm\in I_{\gamma}}\cE_{q_j^\pm}.$$
From the above direct sum and the assumption on $\partial_xu$ we see that only the asymptotic component that corresponds to $q_j^\pm=0$ can occur, which completes the proof. 
\end{proof}

We can use the above lemma to control certain interpolation spaces as follows.

\begin{lemma}\label{inter} Let $s\in\mathbb{R}$, $p,q\in(1,\infty)$ and $\theta\in(0,1)$. Then, the following embeddings hold
$$
\cH^{s+2\theta+\varepsilon,\gg+2\theta+\varepsilon}_p(\B)+\mathbb{C}\hookrightarrow(\cH^{s,\gg}_p(\B),\cH^{s+2,\gg+2}_p(\B)+\mathbb{C})_{\theta,q}
\hookrightarrow \cH^{s+2\theta-\varepsilon,\gg+2\theta-\varepsilon}_p(\B)+\mathbb{C},
$$
for every $\gve>0$. 
\end{lemma}
\begin{proof}
Concerning the first embedding we know from Lemma \ref{int2} that 
\begin{gather*}
\cH^{s+2\theta+\varepsilon,\gg+2\theta+\varepsilon}_p(\B)\hookrightarrow(\cH^{s,\gg}_p(\B),\cH^{s+2,\gg+2}_p(\B))_{\theta,q}.
\end{gather*}
By standard properties of interpolation spaces (see e.g. Proposition 4.2.1 in \cite{Tay1}), we also have 
$$
(\cH^{s,\gg}_p(\B),\cH^{s+2,\gg+2}_p(\B))_{\theta,q}\hookrightarrow(\cH^{s,\gg}_p(\B),\cH^{s+2,\gg+2}_p(\B)+\mathbb{C})_{\theta,q}.
$$
Therefore, 
$$
\cH^{s+2\theta+\varepsilon,\gg+2\theta+\varepsilon}_p(\B)\hookrightarrow(\cH^{s,\gg}_p(\B),\cH^{s+2,\gg+2}_p(\B)+\mathbb{C})_{\theta,q}.
$$
Moreover, 
$$
\mathbb{C}\hookrightarrow(\cH^{s,\gg}_p(\B),\cH^{s+2,\gg+2}_p(\B)+\mathbb{C})_{\theta,q}.
$$
Hence, we conclude that 
$$
\cH^{s+2\theta+\varepsilon,\gg+2\theta+\varepsilon}_p(\B)+\mathbb{C}\hookrightarrow(\cH^{s,\gg}_p(\B),\cH^{s+2,\gg+2}_p(\B)+\mathbb{C})_{\theta,q}.
$$

Let us show now the second embedding. For any 
$$
u\in(\cH^{s,\gg}_p(\B),\cH^{s+2,\gg+2}_p(\B)+\mathbb{C})_{\theta,q},
$$ 
$x\partial_xu$ and $\partial_{y_j}u$ belong to 
$$
(\cH^{s-1,\gg}_p(\B),\cH^{s+1,\gg+2}_p(\B))_{\theta,q}\hookrightarrow \cH^{s+2\theta-1-\varepsilon,\gg+2\theta-\varepsilon}_p(\B),
$$
for any $\varepsilon>0$ by Lemma \ref{int}. Then Lemma \ref{lk2} and \eqref{dmin} show the assertion.
\end{proof}

In this way, we gain some information about the domains of the complex powers of $\underline{\Delta}_{s}$ given in \eqref{dds}. From \cite[I.(2.9.6)]{Am}, \cite[I.(2.5.2)]{Am} and the previous lemma, we obtain the following.

\begin{corollary}\label{complex}
Let $s\in\mathbb{R}$, $\gamma$ as in \eqref{choicegg}, $p\in(1,\infty)$, $0<\Re z<1$ and $c>0$. For the closed extension $\underline{\Delta}_{s}$ given in \eqref{dds}, we then have 
$$
 \cH^{s+2\Re z+\varepsilon,\gg+2\Re z+\varepsilon}_p(\B)+\mathbb{C}\hookrightarrow\mathcal{D}((c-\underline{\Delta}_{s})^{z})\hookrightarrow 
 \cH^{s+2\Re z-\varepsilon,\gg+2\Re z-\varepsilon}_p(\B)+\mathbb{C} \quad \text{for each }\ \gve>0.
$$
\end{corollary}

\section{The Porous Medium Equation on Manifolds with Cones}\setcounter{equation}{0}

We will treat the porous medium equation as a quasilinear problem. Using the multiplier properties of the Mellin-Sobolev spaces of sufficiently large order and weight, we will apply the maximal regularity theorem by Cl\'ement and Li. 
We denote by $(g^{ij})=(g_{ij})^{-1}$ and $(h^{ij})=(h_{ij})^{-1}$ the inverses of the metric tensors of $g$ and $h$ in local coordinates. Then the following identity holds
\begin{gather*}
\Delta(u^{m})=mu^{m-1}\Delta u +m(m-1)u^{m-2}\langle \nabla u,\nabla u\rangle_{g},
\end{gather*}
where
\begin{gather*}
\nabla u=\sum_{i,j}g^{ij}\frac{\partial u}{\partial x^{i}}\frac{\partial}{\partial x^{j}},
\intertext{and, in local coordinates $(x,y)$ near the boundary, } 
\langle \nabla u,\nabla v\rangle_{g}=\frac{1}{x^{2}}\Big((x\partial_{x}u)(x\partial_{x}v)+\sum_{i,j}h^{ij}\frac{\partial u}{\partial y^{i}} \frac{\partial v}{\partial y^{j}} \Big).
\end{gather*}
Equations (\ref{e1}), (\ref{e2}) then take the quasilinear form 
\begin{eqnarray}\label{e3}
u'(t)-m u^{m-1}(t)\Delta u(t)&=&f(u,t)+m(m-1)u^{m-2}(t)\langle \nabla u(t),\nabla u(t)\rangle_{g} \\\label{e4}
u(0)&=&u_{0}.
\end{eqnarray}

In the following, we will show maximal $L^{p}$-regularity for the linearized term of the problem for appropriate initial data. 
The proof is inspired by Theorem 5.7 in \cite{DHP}. We first construct a parametrix with the help of a suitable partition of the space and a Neumann series argument. Then we apply $R$-sectoriality perturbation arguments. 
 
\begin{theorem}\label{t21}
Let $\underline{\Delta}_s$ be chosen as in \eqref{dds} and the weight $\gamma$ as in \eqref{choicegg}. 
We assume that $s\ge0$ and that $p$ and $q$ are so large that $ \frac{n+1}p+\frac{2}{q}<1$ and $\gamma>\frac{n-3}{2}+\frac{2}{q}$. 
Then for any $u\in X_{\frac{1}{q},q}:=(X_1,X_0)_{\frac{1}{q},q}$, such that $u\geq \alpha>0$ on $\mathbb{B}$ and sufficiently large $c>0$, $c-u\underline{\Delta}_s: X_1 \rightarrow X_0$ is a well defined closed linear operator that is $R$-sectorial of angle $\theta$ for any $\theta$ in $[0,\pi)$.

The statement extends to the case where $-1+\frac{n+1}{p}+\frac{2}{q}<s<0$ with $p,q$ as before.
\end{theorem}
\begin{proof}
Recall that $X_{0}=\mathcal{H}_{p}^{s,\gamma}(\mathbb{B})$ and $X_{1}=\mathcal{H}_{p}^{s+2,\gamma+2}(\mathbb{B})\oplus\mathbb{C}$. By Lemma \ref{inter},
\begin{gather}\label{embend}
X_{\frac{1}{q},q} \hookrightarrow \mathcal{H}_{p}^{s+2-\frac{2}{q}-\varepsilon,\gamma+2-\frac{2}{q}-\varepsilon}(\mathbb{B})\oplus\mathbb{C}
\end{gather}
for any $\varepsilon>0$. 
In the sequel, we will only use this property of $u$. 

Since $s+2-\frac{2}{q}>|s|+\frac{n+1}{p}$ and $\gg+2-\frac2q>\frac{n+1}2$, Lemma \ref{c0} implies that $u$ is continuous and constant on $\partial \B$;
according to Corollary \ref{multiplier}, multiplication by $u$ induces a bounded map in $\mathcal{H}_{p}^{s,\gamma}(\mathbb{B})$. Thus, $u\underline{\Delta}_s$ is a well defined closed linear operator in $X_0$ with domain $X_{1}$.

For ${\bar z}\in\mathbb{B}$, let $u({\bar z})\underline{\Delta}_s$ be the realization of $u({\bar z})\Delta$ in $X_{0}$ with domain $X_{1}$. 
For every $c'>0$, $c'-\underline{\Delta}_s$ is $R$-sectorial of angle $\theta$, for each $\theta\in[0,\pi)$ 
by Theorem \ref{t11} and \cite[Theorem 5]{CP}. Fix $c,c'>0$ with $c'\leq c/\|u\|_\infty$.
Denote by $\epsilon_{\rho}$ the $\rho$-th Rademacher function.
The $R$-boundedness of $c'-\underline{\Delta}_{s}$, cf. \eqref{RS1}, \eqref{RS2}, implies that for any $\lambda_{1},...,\lambda_{K}\in\{z\in\mathbb{C}\, |\, |\arg z|\leq \theta\}$, $K\in\mathbb{N}$, and $x_{1},...,x_{K}\in X_{0}$, 
\begin{eqnarray}\nonumber
\lefteqn{\Big\|\sum_{\rho=1}^{K}\lambda_{\rho}(c-u({\bar z})\underline{\Delta}_{s}+\lambda_{\rho})^{-1}\epsilon_{\rho}x_{\rho}\Big\|_{L^{2}(0,1;X_{0})}}\\\nonumber
&=&\Big\|\sum_{\rho=1}^{K}\Big(\frac{c}{u({\bar z})}-c'+\frac{\lambda_{\rho}}{u({\bar z})}-\frac{c}{u({\bar z})}+c'\Big)\Big(c'-\underline{\Delta}_{s}+\frac{c}{u({\bar z})}-c'+\frac{\lambda_{\rho}}{u({\bar z})}\Big)^{-1}\epsilon_{\rho}x_{\rho}\Big\|_{L^{2}(0,1;X_{0})}\\\nonumber
&\leq&\Big\|\sum_{\rho=1}^{K}\Big(\frac{c}{u({\bar z})}-c'+\frac{\lambda_{\rho}}{u({\bar z})}\Big)\Big(c'-\underline{\Delta}_{s}+\frac{c}{u({\bar z})}-c'+\frac{\lambda_{\rho}}{u({\bar z})}\Big)^{-1}\epsilon_{\rho}x_{\rho}\Big\|_{L^{2}(0,1;X_{0})}\\\label{RB}
&&+\Big\|\Big(\frac{c}{u({\bar z})}-c'\Big)(c'-\underline{\Delta}_{s})^{-1}\sum_{\rho=1}^{K}(c'-\underline{\Delta})\Big(c'-\underline{\Delta}_{s}+\frac{c}{u({\bar z})}-c'+\frac{\lambda_{\rho}}{u({\bar z})}\Big)^{-1}\epsilon_{\rho}x_{\rho}\Big\|_{L^{2}(0,1;X_{0})}\\\nonumber
&\leq&\Big(C'(c')+\big(\frac{c}{u({\bar z})}-c'\big)\|(c'-\underline{\Delta}_{s})^{-1}\|_{\mathcal{L}(X_{0})}C''(c')\Big)\Big\|\sum_{\rho=1}^{K}\epsilon_{\rho}x_{\rho}\Big\|_{L^{2}(0,1;X_{0})},\nonumber
\end{eqnarray}
for constants $C'(c'),C''(c')>0$ that depend only on $c'$. Hence, for any $c>0$, $c-u({\bar z})\underline{\Delta}_{s}$ is $R$-sectorial of angle $\theta$ and the $R$-bound is uniform with respect to ${\bar z}$. 
By a well-known perturbation result, see Theorem 1 in \cite{KL}, 
the sum $c-u({\bar z})\underline{\Delta}_{s}-v(z) \underline{\Delta}_{s}$
is again $R$-sectorial of angle $\theta$ with a slightly larger $R$-bound uniformly in ${\bar z}$, provided the norm of $v$ as a multiplier in $X_{0}$ is sufficiently small, say $\|v\cdot\|_{\mathcal{L}(X_{0})}\le\eta$. Here we use the fact that multiplication by $v$ induces a bounded map in $X_0$, by \eqref{embend} and Lemma \ref{c0}.

In order to establish the $R$-sectoriality of $c-u\underline{\Delta}_s$ assume first that $s=0$. 
For $r>0$ choose an open cover of $\mathbb{B}$, consisting of balls $B_j=B_{r}(z_{j})$, $z_{j}\in\mathbb{B}^\circ$, $j\in\{1,...,N\}$, not intersecting the boundary, together with a collar neighborhood $B_0= [0,r)\times\partial \B$. We may assume that also the balls $\overline{B_{3r/2}(z_j)}$ do not intersect $\partial \B$. Let $\omega :\mathbb{R}\rightarrow [0,1]$ be a smooth non-increasing function that equals $1$ on $[0,1/2]$ and $0$ on $[3/4,\infty)$.
Denote by $d=d(z,z')$ the geodesic distance between two points $z,z'\in\mathbb{B}$ with respect to the metric $g$. We define 
\begin{eqnarray*}
u_j(z) &=&\omega\left(\frac{d(z,z_j)}{2r}\right)u(z) + \left(1-\omega\left(\frac{d(z,z_j)}{2r}\right)\right)u(z_j), \quad j=1,\ldots, N,\ \text{and}\\
u_0(z) &=&\omega\left(\frac{d(z,\partial\B)}{2r}\right)u(z) + \left(1-\omega\left(\frac{d(z,\partial \B)}{2r}\right)\right) u(z_0)
\end{eqnarray*}
for some $z_0\in \partial \B$ and write 
\begin{gather*}
c-u_{j}(z)\underline{\Delta}_{s}=c-u(z_{j})\underline{\Delta}_{s}+(u(z_{j})-u_{j}(z))\underline{\Delta}_{s},
\end{gather*}
where each realization is considered with the same domain $X_1$ as $\underline{\Delta}_{s}$. 
Since $\|u(z_{j})-u_{j}(z)\|_{C(\mathbb{B})}$, and therefore the norm of 
$u(z_{j})-u_{j}(z)$ as a multiplier on $\cH^{0,\gamma}_p(\B)$, can be made arbitrarily small by taking $r$ small (for $j=0$ recall that $u$ is constant along the boundary), 
each $c-u_{j}\underline{\Delta}_{s}$ is $R$-sectorial of angle $\theta$. 

Fix a partition of unity $\phi_{j}\in C^{\infty}(\mathbb{B})$, $j=0,\ldots,N$, subordinate to the $B_j$ and functions $\psi_{j}\in C^{\infty}(\mathbb{B})$, supported in $B_j$, 
with $\psi_{j}=1$ on $\mathrm{supp}\, \phi_{j}$. 

Let $f\in X_{1}$, $g\in X_{0}$ and $\lambda\in\mathbb{C}$. 
Multiplying 
\begin{gather*}
\lambda f-u\underline{\Delta}_{s}f=g
\end{gather*}
by $\phi_{j}$ we find 
\begin{gather*}
\lambda \phi_{j}f-u\underline{\Delta}_{s}(\phi_{j}f)=\phi_{j}g-[u\underline{\Delta}_{s},\phi_{j}]f.
\end{gather*}
Applying the resolvent of $u_{j}\underline{\Delta}_{s}$ we get
\begin{gather*}
\phi_{j}f=(\lambda-u_{j}\underline{\Delta}_{s})^{-1}\phi_{j}g-(\lambda-u_{j}\underline{\Delta}_{s})^{-1}[u\underline{\Delta}_{s},\phi_{j}]f.
\end{gather*}
In view of the choice of $\psi_{j}$, the above equation becomes
\begin{gather*}
\phi_{j}f=\psi_{j}(\lambda-u_{j}\underline{\Delta}_{s})^{-1}\phi_{j}g-\psi_{j}(\lambda-u_{j}\underline{\Delta}_{s})^{-1}[u\underline{\Delta}_{s},\phi_{j}]f.
\end{gather*}
Summing up we obtain
\begin{gather}\label{inv1}
f=\sum_{j=0}^{N}\psi_{j}(\lambda-u_{j}\underline{\Delta}_{s})^{-1}\phi_{j}g-\sum_{j=0}^{N}\psi_{j}(\lambda-u_{j}\underline{\Delta}_{s})^{-1}[u\underline{\Delta}_{s},\phi_{j}]f.
\end{gather}
The commutator $[u\gD_{s},\phi_j]$ is a first order cone differential operator with non-smooth coefficients. 
In view of \eqref{embend} and the subsequent considerations, it maps $\cD(\underline\gD_{s})$ to $ \cH_p^{s+\gd,\gg+\gd}(\B)$ for suitably small $\gd>0$ (note that $\phi_0$ is constant near the boundary).
According to Corollary \ref{complex}, the latter space embeds into 
$\mathcal{D}((c-\underline{\Delta}_{s})^{\nu})$ for any $\nu<\gd/2$. 
We may now apply Corollary \ref{Nikos2} and find that 
\begin{eqnarray}\label{commutator}
(\lambda-u_{j}\underline{\Delta}_{s})^{-1}[u\underline{\Delta}_{s},\phi_{j}]\to 0,\quad \gl\to\infty,
\end{eqnarray}
in $\cL(X_1)$.
Hence, $\lambda-u\underline{\Delta}_{s}$ will be left invertible whenever $\lambda\notin\mathbb{R}_{\leq0}$ is sufficiently large. 

We denote the left inverse by $S_{\lambda}$. Then
\begin{eqnarray*}
\lefteqn{(\lambda-u\underline{\Delta}_{s})S_{\lambda}=(\lambda-u\underline{\Delta}_{s})\sum_{j=0}^{N}\psi_{j}(\lambda-u_{j}\underline{\Delta}_{s})^{-1}(\phi_{j}-[u\underline{\Delta}_{s},\phi_{j}]S_{\lambda})}\\
&=&\sum_{j=0}^{N}\psi_{j}(\phi_{j}-[u\underline{\Delta}_{s},\phi_{j}]S_{\lambda})-\sum_{j=0}^{N}[u\underline{\Delta}_{s},\psi_{j}](\lambda-u_{j}\underline{\Delta}_{s})^{-1}(\phi_{j}-[u\underline{\Delta}_{s},\phi_{j}]S_{\lambda}).
\end{eqnarray*}
Since $\sum_{j=0}^{N}\phi_{j}=1$ and $\sum_{j=0}^{N}[u\underline{\Delta}_{s},\phi_{j}]=0$, we conclude that, on $X_0$, 
\begin{gather}\label{e781}
(\lambda-u\underline{\Delta}_{s})S_{\lambda}=I-\sum_{j=0}^{N}[u\underline{\Delta}_{s},\psi_{j}](\lambda-u_{j}\underline{\Delta}_{s})^{-1}(\phi_{j}-[u\underline{\Delta}_{s},\phi_{j}]S_{\lambda}).
\end{gather}
We will next argue that the sum on the right hand side tends to zero in $\cL(X_0)$ as $\gl\to\infty$. Choose $1/2<\nu'<1$ and write
$$
(\lambda-u_{j}\underline{\Delta}_{s})^{-1}
= (c-u_j\underline\gD_{s})^{-\nu'} (c-u_j\underline\gD_{s})
(\gl-u_j\underline\gD_{s})^{-1}(c-u_j\underline\gD_{s})^{-1+\nu'}.
$$ 
Then the product of the last three terms tends to zero in $\cL(X_0)$ as $\gl\to\infty$ by Corollary 
\ref{Nikos2}. 
The operator $(c-u_i\underline\gD_{s})^{-\nu'}$ maps $X_0$ into $\cD((c-u_i\underline\gD_{s})^{\nu'})
\hookrightarrow \cH_p^{s+1,\gg+1}(\B)\oplus\mathbb{C}$ by Corollary \ref{complex}. 
As the commutator $[u\underline{\Delta}_{s},\psi_{j}]$ 
is a cone differential operator of order one it maps $\cH_p^{s+1,\gg+1}(\B)\oplus\mathbb{C}$ continuously to $X_{0}$.
Hence the sum on the right hand side of \eqref{e781} tends to zero and thus
$\lambda-u\underline{\Delta}_{s}$ has also a right inverse for large $\lambda\notin\mathbb{R}_{\leq0}$. 
Let
\begin{gather*}
R(\lambda)=\sum_{j=0}^{N}\psi_{j}(\lambda-u_{j}\underline{\Delta}_{s})^{-1}\phi_{j} \quad \mbox{and} \quad 
Q(\lambda)=\sum_{j={0}}^{N}\psi_{j}(\lambda-u_{j}\underline{\Delta}_{s})^{-1}[u\underline{\Delta}_{s},\phi_{j}].
\end{gather*}
As $Q(\lambda)\to0$ by \eqref{commutator}, we conclude from \eqref{inv1} via a Neumann series argument that, for $\lambda$ large, 
\begin{gather}\label{e551}
(\lambda-u\underline{\Delta}_{s})^{-1}=\sum_{k=0}^{\infty}(-1)^{k}Q^{k}(\lambda)R(\lambda).
\end{gather}
Starting from \eqref{e551} and splitting off the term for $k=0$ we have
\begin{eqnarray}
\lefteqn{\Big\|\sum_{\rho=1}^{K}
\lambda_{\rho}(\lambda_{\rho}+c-u\underline{\Delta}_{s})^{-1}x_{\rho}\epsilon_{\rho}\Big\|_{L^{2}(0,1;X_{0})}}\label{start}\\
&\leq&\Big\|\sum_{\rho=1}^{K}\lambda_{\rho}
\Big(\sum_{j=0}^{N}\psi_{j}(\lambda_{\rho}+c-u_{j}\underline{\Delta}_{s})^{-1}\phi_{j}\Big)x_{\rho}\epsilon_{\rho}\Big\|_{L^{2}(0,1;X_{0})}\nonumber\\
&&+\sum_{k=1}^{\infty}\Big\|\sum_{\rho=1}^{K}\lambda_{\rho}Q^{k}(\lambda_\rho+c)
\Big(\sum_{j=0}^{N}\psi_{j}(\lambda_{\rho}+c-u_{j}\underline{\Delta}_{s})^{-1}\phi_{j}\Big)x_{\rho}\epsilon_{\rho}\Big\|_{L^{2}(0,1;X_{0})}.\nonumber
\end{eqnarray}
We then estimate the first term on the right hand side by 
\begin{eqnarray}
\lefteqn{\le \Big\|\sum_{j=0}^{N}\psi_{j}\sum_{\rho=1}^{K}\lambda_{\rho}(\lambda_{\rho}+c-u_{j}\underline{\Delta}_{s})^{-1}\phi_{j}x_{\rho}\epsilon_{\rho}\Big\|_{L^{2}(0,1;X_{0})}}
\nonumber\\
&\le&\sum_{j=0}^{N}\|\psi_{j}\|_{\mathcal{L}(X_0)}\Big\|\sum_{\rho=1}^{K}\lambda_{\rho}(\lambda_{\rho}+c-u_{j}\underline{\Delta}_{s})^{-1}\phi_{j}x_{\rho}\epsilon_{\rho}\Big\|_{L^{2}(0,1;X_{0})}\nonumber\\
&\le& C_1(N+1)\max_{j}
\Big\|\sum_{\rho=1}^{K}\phi_{j}x_{\rho}\epsilon_{\rho}\Big\|_{L^{2}(0,1;X_{0})}\nonumber\\
&\le& 
C_2(N+1)\ \Big\|\sum_{\rho=1}^{K}x_{\rho}\epsilon_{\rho}\Big\|_{L^{2}(0,1;X_{0})},
\label{first}
\end{eqnarray} 
as required. Next we focus on the summands in the second term for $k=1,2,\ldots$. 
They are 
\begin{eqnarray}
\label{t21.9a}
\quad&\le &\sum_{j=0}^{N}\Big\|\sum_{\rho=1}^{K}\lambda_{\rho}Q^{k}(\lambda_\rho+c)
\psi_{j}(\lambda_{\rho}+c-u_{j}\underline{\Delta}_{s})^{-1}\phi_{j}x_{\rho}\epsilon_{\rho}\Big\|_{L^{2}(0,1;X_{0})}.
\end{eqnarray}

Expanding the product, we regroup this into a sum of $k$ factors; the first equals 
\begin{eqnarray}\label{t21.20}
\psi_j \lambda_\rho (\lambda_\rho+c-u_j\underline \Delta_s)^{-1};
\end{eqnarray}
the others are of the form 
\begin{eqnarray*}\label{t.21.21}
[u\underline \Delta_s,\phi_{j}]\psi_{i}(\lambda_\rho+c-u_i\underline \Delta_s)^{-1}
= [u\underline \Delta_s,\phi_{j}]\psi_{i}
(c-u_{i}\underline{\Delta}_{s})^{-1}(c-u_{i}\underline{\Delta}_{s})
(\lambda_\rho+c-u_i\underline \Delta_s)^{-1}
\end{eqnarray*}
for suitable $i,j \in \{0,\ldots, N\}$. Now $[u\underline \Delta_s,\phi_{j}]\psi_{i}$ is a first order cone 
differential operator. By a similar argument as before, the norm of 
$[u\underline \Delta_s,\phi_{j}]\psi_{i}
(c-u_{i}\underline{\Delta}_{s})^{-1}$
in $\cL(X_0)$ will tend to zero as $c\to \infty$. 
Given $\gve>0$, we choose $c$ so large that all these norms are $<\gve$. 
Using successively the $R$-sectoriality of $c-u_j\underline \Delta_s$ together with Lemma \ref{easy} and \eqref{RS2} we can estimate \eqref{t21.9a} by 
$$\le C_3(N+1) ((N+1)C\gve)^{k-1}\Big\|\sum_{\rho=1}^{K}x_{\rho}\epsilon_{\rho}\Big\|_{L^{2}(0,1;X_{0})}$$ 
with a fixed constant $C$.
Summing over $k$ and combining this with the estimate \eqref{first}, we see that
\eqref{start} is 
$$\le C_5\ \frac{N+1}{1-(N+1)C\gve} \ \Big\|\sum_{\rho=1}^{K}x_{\rho}\epsilon_{\rho}\Big\|_{L^{2}(0,1;X_{0})},
$$ 
which establishes the $R$-sectoriality in $S_\theta$. 

Next we will show that the resolvent of $c-u\underline{\Delta}_{s}$ for $s>0$ 
is the restriction of the resolvent of $c-u\underline{\Delta}_{0}$ to $\mathcal{H}_{p}^{s,\gamma}(\mathbb{B})$. It suffices to prove that the following identities 
$$
(u\underline{\Delta}_{0}-\lambda)^{-1}(u\Delta-\lambda)=I \quad \text{and} \quad (u\Delta-\lambda)(u\underline{\Delta}_{0}-\lambda)^{-1}=I
$$
hold in $X_1$ and $X_0$ respectively. Let $x\in\mathcal{D}(u\underline{\Delta}_{0})=\mathcal{H}_{p}^{2,\gamma+2}(\mathbb{B})\oplus\mathbb{C}$ such that $u\Delta x\in \mathcal{H}_{p}^{s,\gamma}(\mathbb{B})$. 

By \eqref{embend}, $X_{\frac{1}{q},q}$ embeds into $ \mathcal{H}_{p}^{s+2-\frac{2}{q}-\varepsilon,\gamma+2-\frac{2}{q}-\varepsilon}(\mathbb{B})\oplus\mathbb{C}$, for all $\varepsilon>0$, which is a Banach algebra. 
As shown in Lemma \ref{linv}, below, also $u^{-1}$ belongs to $\mathcal{H}_{p}^{s+2-\frac{2}{q}-\varepsilon,\gamma+2-\frac{2}{q}-\varepsilon}(\mathbb{B})\oplus\mathbb{C}$, for all $\varepsilon>0$. 
Since elements in that space are multipliers in $\mathcal{H}_{p}^{s,\gamma}(\mathbb{B})$, we conclude that $\Delta x\in \mathcal{H}_{p}^{s,\gamma}(\mathbb{B})$. 
Therefore, $x\in\mathcal{D}(\underline{\Delta}_{s,\max})\cap(\mathcal{H}_{p}^{2,\gamma+2}(\mathbb{B})\oplus\mathbb{C})$, and hence $x\in\mathcal{H}_{p}^{s+2,\gamma+2}(\mathbb{B})\oplus\mathbb{C}$. 

Let us now treat the case $s\geq1$, $s\in\mathbb{N}$, by induction.
Let $y_0,...,y_{n}$ be local coordinates on the support of $\phi_i$, with the understanding that, in the collar neighborhood of the boundary, we use local coordinates $x,y_1,...,y_n$ and replace the derivative $\partial_{y_0}$ in the computation below by $x\partial_x$. Denote by $M_{\phi_i}$ the multiplication operator by $\phi_i$. Then

\begin{eqnarray}\nonumber
\lefteqn{
\Big\|\sum_{\rho=1}^{K}\lambda_{\rho}(\lambda_{\rho}+c-u\underline{\Delta}_{s+1})^{-1}x_{\rho}\epsilon_{\rho}\Big\|_{L^{2}(0,1;\mathcal{H}_{p}^{s+1,\gamma}(\mathbb{B}))}
=\Big\|\sum_{\rho=1}^{K}\lambda_{\rho}(\lambda_{\rho}+c-u\underline{\Delta}_{s})^{-1}x_{\rho}\epsilon_{\rho}\Big\|_{L^{2}(0,1;\mathcal{H}_{p}^{s+1,\gamma}(\mathbb{B}))}}\\\nonumber 
&\le&\sum_{|a|\leq 1}\Big\|\sum_{j=0}^{N}\sum_{\rho=1}^{K}\partial_{y}^{a}M_{\phi_{j}}\lambda_{\rho}(\lambda_{\rho}+c-u\underline{\Delta}_{s})^{-1}x_{\rho}\epsilon_{\rho}\Big\|_{L^{2}(0,1;\mathcal{H}_{p}^{s,\gamma}(\mathbb{B}))}\\\nonumber
&\leq&\sum_{|a|\leq 1}\Big\|\sum_{j=0}^{N}\sum_{\rho=1}^{K}\lambda_{\rho}(\lambda_{\rho}+c-u\underline{\Delta}_{s})^{-1}\partial_{y}^{a}M_{\phi_{j}}x_{\rho}\epsilon_{\rho}\Big\|_{L^{2}(0,1;\mathcal{H}_{p}^{s,\gamma}(\mathbb{B}))}
\\\nonumber
&&\hspace{11.5em}+\sum_{|a|\leq 1}\Big\|\sum_{j=0}^{N}\sum_{\rho=1}^{K}\lambda_{\rho}[\partial_{y}^{a}M_{\phi_{j}},(\lambda_{\rho}+c-u\underline{\Delta}_{s})^{-1}]x_{\rho}\epsilon_{\rho}\Big\|_{L^{2}(0,1;\mathcal{H}_{p}^{s,\gamma}(\mathbb{B}))}\\\nonumber
&=&\sum_{|a|\leq 1}\Big\|\sum_{\rho=1}^{K}\sum_{j=0}^{N}\lambda_{\rho}(\lambda_{\rho}+c-u\underline{\Delta}_{s})^{-1}\partial_{y}^{a}M_{\phi_{j}}x_{\rho}\epsilon_{\rho}\Big\|_{L^{2}(0,1;\mathcal{H}_{p}^{s,\gamma}(\mathbb{B}))}\\\nonumber
&&+\sum_{|a|\leq 1}\Big\|\sum_{\rho=1}^{K}\sum_{j=0}^{N}\lambda_{\rho}(\lambda_{\rho}+c-u\underline{\Delta}_{s})^{-1}\\\label{ltr}
&&\hspace{5em}\times[\partial_{y}^{a}M_{\phi_{j}},c-u\underline{\Delta}_{s}](c-u\underline{\Delta}_{s})^{-1}(c-u\underline{\Delta}_{s})(\lambda_{\rho}+c-u\underline{\Delta}_{s})^{-1}x_{\rho}\epsilon_{\rho}\Big\|_{L^{2}(0,1;\mathcal{H}_{p}^{s,\gamma}(\mathbb{B}))}.
\end{eqnarray}

By the assumption on $q$ and $p$, the coefficients of the second order cone differential operator $[\partial_{y}^{a}M_{\phi_{i}},c-u\underline{\Delta}_{s}]$ are multipliers in $\mathcal{H}_{p}^{s,\gamma}(\mathbb{B})$ (note that by the induction hypothesis now $u$ belongs to $\mathcal{H}_{p}^{s+3-\frac{2}{q}-\varepsilon,\frac{n+1}{2}}(\mathbb{B})\oplus\mathbb{C}$). 
Therefore, the operators $[\partial_{y}^{a}M_{\phi_{i}},c-u\underline{\Delta}_{s}](c-u\underline{\Delta}_{s})^{-1}$ are bounded maps in $\mathcal{H}_{p}^{s,\gamma}(\mathbb{B})$. Using the $R$-sectoriality of $c-u\underline{\Delta}_{s}$ in $\mathcal{H}_{p}^{s,\gamma}(\mathbb{B})$, we obtain from \eqref{ltr} 
\begin{eqnarray*}
\lefteqn{\Big\|\sum_{\rho=1}^{K}\lambda_{\rho}(\lambda_{\rho}+c-u\underline{\Delta}_{s+1})^{-1}x_{\rho}\epsilon_{\rho}\Big\|_{L^{2}(0,1;\mathcal{H}_{p}^{s+1,\gamma}(\mathbb{B}))}}\\\nonumber
&\leq&C_6 \sum_{|a|\leq 1}\Big\|\sum_{\rho=1}^{K}\sum_{j=0}^{N}\partial_{y}^{a}M_{\phi_{j}}x_{\rho}\epsilon_{\rho}\Big\|_{L^{2}(0,1;\mathcal{H}_{p}^{s,\gamma}(\mathbb{B}))}
+C_6\sum_{|a|\leq 1}\Big\|\sum_{\rho=1}^{K}\sum_{j=0}^{N}[\partial_{y}^{a}M_{\phi_{j}},c-u\underline{\Delta}_{s}](c-u\underline{\Delta}_{s})^{-1}\\\nonumber
&&\hspace{19em}\times(c-u\underline{\Delta}_{s})(\lambda_{\rho}+c-u\underline{\Delta}_{s})^{-1}x_{\rho}\epsilon_{\rho}\Big\|_{L^{2}(0,1;\mathcal{H}_{p}^{s,\gamma}(\mathbb{B}))}
\\\nonumber
&\leq&C_7 \Big\|\sum_{\rho=1}^{K}x_{\rho}\epsilon_{\rho}\Big\|_{L^{2}(0,1;\mathcal{H}_{p}^{s+1,\gamma}(\mathbb{B}))}
+C_7\max_{j,a}\Big\|[\partial_{y}^{a}M_{\phi_{j}},c-u\underline{\Delta}_{s}](c-u\underline{\Delta}_{s})^{-1}\Big\|_{\mathcal{L}(\mathcal{H}_{p}^{s,\gamma}(\mathbb{B}))}\\
&&\hspace{16.5em}\times\Big\|\sum_{\rho=1}^{K}(c-u\underline{\Delta}_{s})(\lambda_{\rho}+c-u\underline{\Delta}_{s})^{-1}x_{\rho}\epsilon_{\rho}\Big\|_{L^{2}(0,1;\mathcal{H}_{p}^{s,\gamma}(\mathbb{B}))}
\end{eqnarray*}
\begin{eqnarray*}
&\leq&C_7 \Big\|\sum_{\rho=1}^{K}x_{\rho}\epsilon_{\rho}\Big\|_{L^{2}(0,1;\mathcal{H}_{p}^{s+1,\gamma}(\mathbb{B}))}+C_8\Big\|\sum_{\rho=1}^{K}x_{\rho}\epsilon_{\rho}\Big\|_{L^{2}(0,1;\mathcal{H}_{p}^{s,\gamma}(\mathbb{B}))}\\\nonumber
&\leq&C_9 \Big\|\sum_{\rho=1}^{K}x_{\rho}\epsilon_{\rho}\Big\|_{L^{2}(0,1;\mathcal{H}_{p}^{s+1,\gamma}(\mathbb{B}))},
\end{eqnarray*}
for suitable constants $C_6$, $C_7$, $C_8$ and $C_9$.

The case where $s\in\mathbb{R}$, $s\geq0$, follows by interpolation, see Lemma \ref{sharpint} and \cite[Theorem 3.19]{KS}.

Assume finally that $-1+\frac{n+1}{p}+\frac{2}{q}<s<0$. We basically proceed as in the case $s=0$. The crucial step is to prove that the norm of $u(z_{j})-u_{j}(z)$ as a multiplier in $\mathcal{H}_{p}^{s,\gamma}(\mathbb{B})$ can be made arbitrarily small by choosing the radius $r$ of the covering small. 
In view of the fact that $|s|<1-(n+1)/p$, Corollary \ref{multiplier} implies that it is sufficient to show that the norm of $u(z_{j})-u_{j}(z)$ in 
$\cH_p^{1-\gve, (n+1)/2}(\B)$ tends to zero as $r\to0$ for any fixed 
$\gve>0$ (note that the constant part of $u$ cancels when we take the difference). 
In view of the fact that our assumptions on $s,p$ and $q$ imply that $u\in \cH^{1+(n+1)/p+\gd,(n+1)/2+\gd}_p(\B)$ for suitably small $\gd>0$, this will follow form a standard interpolation inequality, see e.g. Proposition I.2.2.1 in \cite{Am}, provided we show that 
\begin{eqnarray}\label{0norm}
\|u(z_j)-u_j(z)\|_{\cH^{0,(n+1)/2}_p(\B)} && \text{tends to zero as }r\to0;\\
\label{1norm}
\|u(z_j)-u_j(z)\|_{\cH^{1,(n+1)/2}_p(\B)}&&\text{is bounded as } r\to 0.
\end{eqnarray}
Now \eqref{0norm} is obvious. As for \eqref{1norm} write
\begin{eqnarray}\label{t21.25}
u_j(z) - u(z_j) = \go\Big(\frac{d(z,z_j)}{2r}\Big) (u(z)-u(z_j)).
\end{eqnarray}
Being an element of $\cH_p^{1+(n+1)/p+\gd,(n+1)/2+\gd}(\B)$ for some $\gd>0$, $u$ is Lipschitz continuous in the interior of $\B$, so that 
$$|u(z)-u(z_j)|\le L d(z,z_j)$$ 
for some $L\ge0$. In view of the fact that $\go'\Big(\frac{d}{2r}\Big)\frac{d}{2r}$
is bounded uniformly in $r$, we obtain a uniform bound on the derivatives in \eqref{t21.25}
on the balls in the interior. 
Concerning $B_0$, we notice that, for $z=(x,y)$ we can take $d(z,\partial \B)=x$ in the definition of $u_0$ and use the fact that $\partial_y\go\Big(\frac{x}{2r}\Big) =0$, while 
$x\partial_x \go\Big(\frac{x}{2r}\Big) = \go'\Big(\frac{x}{2r}\Big) \frac x{2r}$ is again bounded.
This completes the argument. 
 \end{proof}

\begin{lemma}\label{linv}
Let $u\in \mathcal{H}_{p}^{s_0,\gamma_0}(\mathbb{B})\oplus\mathbb{C}$, for some $s_0>\frac{n+1}{p}$, $\gamma_0>\frac{n+1}{2}$. If $u$ is pointwise invertible, then $u^{-1}\in \mathcal{H}_{p}^{s_0,\gamma_0}(\mathbb{B})\oplus\mathbb{C}$. As a consequence, $\mathcal{H}_{p}^{s_0,\gamma_0}(\mathbb{B})\oplus\mathbb{C}$ is spectrally invariant in $C(\B)$ and therefore closed under holomorphic functional calculus. 
\end{lemma}
\begin{proof} 
The case, where $s_0$ is a positive integer, is straightforward. 
Otherwise we choose the $B_j$, $u_j$ and $\phi_j$ as in the proof of Theorem \ref{t21}.
We assume that $r$ is chosen so small that 
$$\|u_j(z) -u(z_j)\|_{\sup} \le \frac12 |u(z_j)|.$$

In an initial step, we will show that the inverses of the $u_j$ have the asserted property. We consider first an interior ball $B_j$, $j\ge1$. 
We may assume that even $\overline {B_{3r/2}(z_j)}$ is contained in single coordinate neighborhood for $\B^\circ$. The function $v_j(z)=u_j(z) -u(z_j)$ is supported there, and its push-forward under the coordinate chart belongs to $H^{s_0}_{p}(\R^{n+1})$. 
To simplify the computation, below, suppose that $u(z_j)=1$. 
Theorems 6 and 10 in \cite{BS} show that $v_j(1+v_j)^{-1}$ also belongs to $H^{s_0}_{p}(\R^{n+1})$. Its pullback under the coordinate map then is an element of $\mathcal{H}_{p}^{s_0,\gamma_0}(\mathbb{B})$. 
Since 
$$u_j^{-1} = (1+v_j)^{-1} = 1-v_j(1+v_j)^{-1}$$
we conclude that $u_j^{-1} $ belongs to $\mathcal{H}_{p}^{s_0,\gamma_0}(\mathbb{B})\oplus \C$. 

As for $v_0$, we first note that it also is an element of $\mathcal{H}_{p}^{s_0,(n+1)/2}(\mathbb{B})$. 
Using possibly a further partition of unity, we may assume that it is supported in a single boundary chart. Denote by $v_*$ its push-forward under the coordinate map. According to Definition \ref{dms}, $V= \mathcal S_{(n+1)/2}v_*$ is an element of $H^{s_0}_p(\R^{n+1})$. 
By the above theorems, the same is true for $V(1+V)^{-1}$, and we deduce that $v_0(1+v_0)^{-1}$ is an
element of $\mathcal{H}_{p}^{s_0,(n+1)/2}(\mathbb{B})$. 
Now we write 
$$u_0^{-1} = (1+v_0)^{-1} = 1-v_0+ v_0\big(v_0(1+v_0)^{-1}\big).$$
Lemma \ref{c0} then shows that the right hand side is in $\mathcal{H}_{p}^{s_0,\gamma_0}(\mathbb{B})\oplus \C$. 

To see that $u^{-1}$ also has the asserted property, we recall that $u_j$ coincides with $u$ on the support of $\phi_j$, so that 
\begin{equation}\label{inv.1}
1=\sum \phi_j = \sum \phi_j u_j^{-1} u .
\end{equation}
This shows the spectral invariance of $\mathcal{H}_{p}^{s_0,\gamma_0}(\mathbb{B})\oplus \C$ in $C(\B)$. 

It is well-known that this implies the closedness under holomorphic functional calculus. We recall the argument for the convenience of the reader. 
If $h$ is a holomorphic function in a neighborhood of $ \mathrm{Ran}(-u)$, then
$$
h(u)=\frac{1}{2\pi i}\int_{\Gamma}h(-\lambda)(u+\lambda)^{-1}d\lambda, 
$$
where $\Gamma$ is a finite path around $\mathrm{Ran}(-u)$. Since the spectral invariance implies the continuity of inversion, we conclude that $h(u)\in\mathcal{H}_{p}^{s_0,\gamma_0}(\mathbb{B})\oplus\mathbb{C}$.
\end{proof}

\begin{lemma}\label{boundinv}Under the conditions of Lemma \ref{linv} let
$U$ be a bounded open subset of $\mathcal{H}_{p}^{s_0,\gamma_0}(\mathbb{B})\oplus\mathbb{C}$ consisting of functions $u$ such that $\Re u\geq\alpha>0$ for some fixed $\alpha$. 
Then the subset $\{u^{-1}\, |\, u\in U\}$ of $\mathcal{H}_{p}^{s_0,\gamma_0}(\mathbb{B})\oplus\mathbb{C}$ is also bounded. 
\end{lemma}
\begin{proof}
If $s_0$ is a positive integer, this is straightforward from the identity $\partial_y (u^{-1}) = 
u^{-2}\partial_y u$. For non-integer $s_0$ we transfer the problem to $\R^{n+1}$ as in the proof of Lemma \ref{linv} and apply \cite[Theorem 11]{BS}. 
\end{proof}

\begin{remark}\label{rem}
Let $\nu\in \R$ and $\theta$ in $[0,\pi)$. 
Under the conditions of Theorem $\ref{t21}$, the operator $c-u^\nu \underline{\Delta}_s$ is a closed linear operator on $X_0$ with domain $X_1$ which is $R$-sectorial of angle $\theta$, provided $c>0$ is sufficiently large. 

In fact, since $u$ belongs to 
$\mathcal{H}_{p}^{s+2-\frac{2}{q}-\varepsilon,\gamma+2-\frac{2}{q}-\varepsilon}(\mathbb{B})\oplus\mathbb{C}$ for each $\varepsilon>0$, the same is true for $u^\nu$ by Lemma $\ref{linv}$. As pointed out in the beginning of the proof of Theorem \ref{t21}, this is all we need to know about $u$ for the proof of the $R$-sectoriality.
\end{remark}

Summarizing what we have found we obtain the following result for the porous medium equation.

\begin{theorem}\label{pmt} 
Let $\lambda_{1}$ be the 
largest nonzero eigenvalue of the boundary Laplacian 
$\Delta_{\partial}$, induced by the metric $h$ on $\partial\mathbb{B}$. 
Recall that $\dim(\B)=n+1$, that
$$\overline\gve_n=-\frac{n-1}2+\sqrt{\left(\frac{n-1}2\right)^2-\lambda_1}\ >0,
$$
and that $\gg$ satisfies \eqref{choicegg}, i.e. 
\begin{eqnarray*}
\frac{n-3}2<\gg<\min\left\{\frac{n-3}2+\overline\gve_n ,\frac{n+1}2\right\}.
\end{eqnarray*}
Next choose $p,q$ so large that 
\begin{eqnarray}\label{pq1}
\frac{n+1}p+\frac2q<1 \quad\text{and}\quad \gamma>\frac{n-3}{2}+\frac{4}{q}.
\end{eqnarray}
Then, for any $s>-1+\frac{n+1}{p}+\frac{2}{q}$ and for any strictly positive initial value 
$$
u_{0}\in(\mathcal{H}_{p}^{s+2,\gamma+2}(\mathbb{B})\oplus\mathbb{C},\mathcal{H}_{p}^{s,\gamma}(\mathbb{B}))_{\frac{1}{q},q}
$$
there exists some $0<T\leq T_{0}$ such that the porous medium equation \eqref{e1}, \eqref{e2} in the space 
$L^{q}(0,T;\mathcal{H}_{p}^{s,\gamma}(\mathbb{B}))$
has a unique solution 
$$
u\in L^{q}(0,T;\mathcal{H}_{p}^{s+2,\gamma+2}(\mathbb{B})\oplus\mathbb{C})\cap W^{1,q}(0,T;\mathcal{H}_{p}^{s,\gamma}(\mathbb{B})).
$$
\end{theorem} 
\begin{proof} 
We shall apply Cl\'ement and Li's Theorem \ref{CL}, to the porous medium equation in the form (\ref{e3}), (\ref{e4}), with $X_{0}=\cH^{s,\gg}_p(\B)$ and $X_{1}=\cH^{s+2,\gg+2}_p(\B)\oplus \C$. 

For each initial value $u_{0}\in X_{{1}/{q},q}=(X_{1},X_{0})_{{1}/{q},q}$ such that $u_{0}> \alpha>0$, the operator $m u_{0}^{m-1}\underline{\Delta}$ has maximal 
$L^q$-regularity by Remark \ref{rem}. It remains to check the conditions (H1) and (H2). 
By our assumptions on $s,p,q$ and $\gamma$, \eqref{embend} in connection with Corollary \ref{multiplier} implies that $X_{{1}/{q},q}$ embeds into a Banach algebra of bounded multipliers on $X_0$. 
So let $U$ be a bounded open neighborhood of $u_{0}$ in $X_{{1}/{q},q}$. Since $(X_{1},X_{0})_{{1}/{q},q}$ embeds into $C(\B)$, we may assume $U$ to consist of functions with real part $\geq \alpha$. For any $u_1,u_2\in U$ and real $\nu$, we have then
\begin{gather}\label{poly}
u_{1}^{\nu}-u_{2}^{\nu}=(u_{2}-u_{1})\frac{1}{2\pi i}\int_{\Gamma}(-\lambda)^{\nu}(u_1+\lambda)^{-1}(u_2+\lambda)^{-1}d\lambda,
\end{gather}
where $\Gamma$ is a fixed finite path around $\cup_{u\in U}\mathrm{Ran}(-u)$ in $\{ \Re\lambda<0\}$.

Concerning (H1): Equation \eqref{poly} in connection with the embedding \eqref{embend}, Lemma \ref{linv}, Lemma \ref{boundinv}, and the boundedness of $U$ implies that 
\begin{eqnarray}\label{h1}
\lefteqn{\|mu_{1}^{m-1}\underline{\Delta}-mu_{2}^{m-1}\underline{\Delta}\|_{\mathcal{L}(X_1,X_0)}\leq\|mu_{1}^{m-1}-mu_{2}^{m-1}\|_{\mathcal{L}(X_0)}}
\nonumber\\
\quad &\leq&C_{1}\|u_1-u_2\|_{\cL(X_0)}\leq C_{2}\|u_1-u_2\|_{X_{\frac{1}{q},q}},
\end{eqnarray}
for suitable constants $C_1,C_2>0$. 

Concerning (H2): Equation \eqref{embend} shows that $\nabla u_1$ and $\nabla u_2$ 
belong to $\cH^{s+1-2/q-\gve,\gg+1-2/q-\gve}_p(\B)$ for each $\gve>0$. As $2/q<1$, they are elements of $X_0$. 
Write 
\begin{eqnarray}
\lefteqn{\|f(u_{1},t_1)+m(m-1)u_{1}^{m-2}\langle \nabla u_{1},\nabla u_{1}\rangle_{g}-f(u_{2},t_2)-m(m-1)u_{2}^{m-2}\langle \nabla u_{2},\nabla u_{2}\rangle_{g}\|_{X_0}}
\nonumber\\
&\leq&\|f(u_{1},t_1)-f(u_{2},t_2)\|_{X_0}+m\, |m-1|\, \|u_{1}^{m-2}\langle \nabla u_{1},\nabla u_{1}\rangle_{g}-u_{2}^{m-2}\langle \nabla u_{2},\nabla u_{2}\rangle_{g}\|_{X_0}
\nonumber\\
&\leq&\|f(u_{1},t_1)-f(u_2,t_1)\|_{X_0}+\|f(u_2,t_1)-f(u_{2},t_2)\|_{X_0}
\nonumber\\
&&+m\, |m-1|\, \|u_{1}^{m-2}\langle \nabla u_{1},\nabla u_{1}\rangle_{g}-u_{2}^{m-2}\langle \nabla u_{1},\nabla u_{1}\rangle_{g}+u_{2}^{m-2}\langle \nabla u_{1},\nabla u_{1}\rangle_{g}
\nonumber\\
&&-u_{2}^{m-2}\langle \nabla u_{2},\nabla u_{1}\rangle_{g}+u_{2}^{m-2}\langle \nabla u_{2},\nabla u_{1}\rangle_{g}-u_{2}^{m-2}\langle \nabla u_{2},\nabla u_{2}\rangle_{g}\|_{X_0},
\label{pmt.1}
\end{eqnarray}
with $t_{1},t_{2}\in(0,T_0)$. By assumption, 
\begin{gather*}
f(u,t)=\frac{1}{2\pi i}\int_{\Gamma}f(-\lambda,t)(u+\lambda)^{-1}d\lambda, \quad u\in U, \,\,\, t\in[0,T_0].
\end{gather*}
Therefore,
\begin{gather}\label{ee1}
f(u_1,t)-f(u_2,t)=(u_2 - u_1)\frac{1}{2\pi i}\int_{\Gamma}f(-\lambda,t)(u_1+\lambda)^{-1}(u_2+\lambda)^{-1}d\lambda, \quad t\in[0,T_0],
\end{gather}
and 
\begin{gather}\label{ee2}
f(u,t_1)-f(u,t_2)=\frac{1}{2\pi i}\int_{\Gamma}\big(f(-\lambda,t_1)-f(-\lambda,t_2)\big)(u+\lambda)^{-1}d\lambda, \quad u\in U.
\end{gather}

Equation \eqref{ee1}, the embedding \eqref{embend} together with Lemma \ref{linv}, Lemma \ref{boundinv}, the boundedness of $U$ and the fact that $X_{1/q,q}$ embeds into a Banach algebra of multipliers on $X_0$ shows that the first term on the right hand side of Equation \eqref{pmt.1} is bounded by 
$c_1\|u_1-u_2\|_{X_{1/q,q}}$ for a suitable constant $c_1$. The second term is bounded by $c_2|t_1-t_2|$ in view of \eqref{ee2} and the Lipschitz continuity of $f(-\lambda,t)$ in $t$.

Neglecting the factor $m\, |m-1|$ we estimate the third term by 
\begin{eqnarray}
\label{pmt.2}
\lefteqn{\|u_1^{m-2}-u_2^{m-2}\|_{\cL(X_0)}\|\skp{\nabla u_1,\nabla u_1}_g\|_{X_0} }\\
&&+ 
\|u_2^{m-2}\|_{\cL(X_0)} \left(\|\skp{\nabla(u_1-u_2),\nabla u_1}_g\|_{X_0} 
+ \|\skp{\nabla(u_1-u_2),\nabla u_2}_g\|_{X_0}\right).\nonumber
\end{eqnarray}
Let $\tilde{x}$ be a smooth function on $\mathbb{B}$ that is equal to $x$ in $[0,1/2]\times\partial\mathbb{B}$ and constant 1 outside $[0,1]\times\partial\mathbb{B}$.
Choose $0<\delta < 1-2/q$ such that $\gamma+1-2/q+\gd>(n+1)/2$, cf. \eqref{pq}.
Given $v_1,v_2\in X_{1/q,q}$ we find for the local partial derivatives that
$\tilde x^\delta\partial v_1 \in \cH^{s+1-2/q-\gve,(n+1)/2}$ for each $\gve>0$ and $\tilde x^{-\delta} \partial v_2 \in X_0$, where, in the collar neighborhood of the boundary with coordinates $(x,y)$ we use $\partial_x$ and $\frac1x \partial_{y_j}$, $j=1,\ldots,n$. 
Hence the norm of $\skp{\nabla v_1,\nabla v_2}_g$ in $X_0$ can be estimated with Corollary \ref{multiplier} by the norms of the $\tilde x^\delta \partial v_1$ as multipliers on $X_0$ and the norms of $x^{-\delta}\partial v_2$ in $X_0$. Thus 
\begin{eqnarray}
\|\skp{\nabla v_1, \nabla v_2}_g\|_{X_0} = \|\skp{\tilde x^\delta \nabla v_1, \tilde x^{-\delta}\nabla v_2}_g\|_{X_0} 
\label{pmt.3}
\le c_1\|v_1\|_{X_{1/q,q}}\|v_2\|_{X_{1/q,q}}
\end{eqnarray}
for a suitable constants $c_0, c_1$. 

This enables us to estimate the first term in \eqref{pmt.2} with the help of Equation \eqref{poly} and the second and third with \eqref{pmt.3} by $c_3\|u_1-u_2\|_{X_{1/q,q}}$ 
for a suitable constant $c_3$ and thus completes the argument. 
\end{proof}

\section{The Equation in Spaces with Asymptotics}\setcounter{equation}{0}

In this section we will obtain more precise statements on the asymptotics of the solutions to the porous medium equation near the tip. 
As before, let $s\in \R$ and $1<p,q<\infty$ with 
\begin{gather}\label{spq}
-1 +\frac{n+1}p+\frac2q<\min\{0,s\}.
\end{gather}
For the time being, we continue to assume that $\gamma$ satisfies \eqref{choicegg} with $\gg>\frac{n-3}2+\frac2q$ and denote by 
$$\underline\gD_s: \cD(\underline\gD_s)= \cH^{s+2,\gg+2}_p(\B)\oplus \C\to \cH^{s,\gamma}_p(\B)$$ 
the unbounded operator in $\cH^{s,\gg}_p(\B)$ associated with these data. 

Below, we will consider the Laplacian as an unbounded operator in $Y_0= \cD(\underline\gD_s)$ with domain $Y_1 = \cD(\underline\gD_s^2)$. 
In order to obtain once more a maximally regular solution of the porous medium equation we will have to suppose that $\dim(\B)\not=3$, to make further 
assumptions on the spectrum of $\gD_\partial$, and to slightly change $\gamma$. 
We start with the following elementary result.

\begin{lemma}\label{scale}
Let $0<\ga\le u$ be a strictly positive 
function in $$ X_{_{\frac{1}{q},q}}:=(\mathcal{H}_{p}^{s+2,\gamma+2}(\mathbb{B})\oplus\mathbb{C},\mathcal{H}_{p}^{s,\gamma}(\mathbb{B}))_{\frac{1}{q},q}.$$ 
We let 
\begin{eqnarray}\label{scale.0}
\cD(u\underline\gD_s) = \cD(\underline\gD_s) = \mathcal{H}_{p}^{s+2,\gamma+2}(\mathbb{B})\oplus\mathbb{C}
\end{eqnarray}
and define recursively for $k\in\mathbb{N}$
$$\mathcal{D}((u\underline{\Delta}_{s})^{k+1}):=\{v\in\mathcal{D}((u\underline{\Delta}_{s})^{k})\, |\,u\underline{\Delta}_{s}v \in\mathcal{D}((u\underline{\Delta}_{s})^{k})\}.$$
Then, for any $\theta\in[0,\pi)$, there exists some $c>0$ such that 
\begin{eqnarray}\label{scale.1}
c-u\underline{\Delta}_{s}:\mathcal{D}((u\underline{\Delta}_{s})^{k+1})\to \mathcal{D}((u\underline{\Delta}_{s})^{k})
\end{eqnarray}
is $R$-sectorial of angle $\theta$.
\end{lemma}
\begin{proof} 
Definition \eqref{scale.0} makes sense, since $u$ is a multiplier on $\cH^{s,\gamma}_p(\B)$ by our choice of $q$. 
As $c-u\underline\Delta_s$ is sectorial on 
$\cD(\underline\gD_s)$, it is also sectorial in 
\eqref{scale.1} for each $k$ by Lemma V.1.2.3 in \cite{Am}. 
For the $R$-sectoriality we proceed by induction starting with Theorem \ref{t21} for $k=0$. 
Assume that the assertion holds for some $k\ge0$. Then, for any $\lambda_{1},...,\lambda_{K}\in\{z\in\mathbb{C}\, |\, |{\arg}z|\leq \theta\}$, $K\in\mathbb{N}$, and $x_{1},...,x_{K}\in \mathcal{D}((u\underline{\Delta}_{s})^{k+1})$, 
\begin{eqnarray*}
\lefteqn{\Big\|\sum_{\rho=1}^{K}\lambda_{\rho}(c-u\underline{\Delta}_{s}+\lambda_\rho)^{-1}\epsilon_{\rho}x_{\rho}\Big\|_{L^{2}(0,1;\mathcal{D}((u\underline{\Delta}_{s})^{k+1}))}}\\
&\le&c_1\Big\|\sum_{\rho=1}^{K}(c-u\underline{\Delta}_{s})\big(\lambda_{\rho}(c-u\underline{\Delta}_{s}+\lambda_\rho)^{-1}\epsilon_{\rho}x_{\rho}\big)\Big\|_{L^{2}(0,1;\mathcal{D}((u\underline{\Delta}_{s})^{k}))}\\
&=&c_1\Big\|\sum_{\rho=1}^{K}\lambda_{\rho}(c-u\underline{\Delta}_{s}+\lambda_\rho)^{-1}\epsilon_{\rho}(c-u\underline{\Delta}_{s})x_{\rho}\Big\|_{L^{2}(0,1;\mathcal{D}((u\underline{\Delta}_{s})^{k}))}\\
&\leq&c_2\Big\|\sum_{\rho=1}^{K}\epsilon_{\rho}(c-u\underline{\Delta}_{s})x_{\rho}\Big\|_{L^{2}(0,1;\mathcal{D}((u\underline{\Delta}_{s})^{k}))}\\
&\le&c_3\Big\|\sum_{\rho=1}^{K}\epsilon_{\rho}x_{\rho}\Big\|_{L^{2}(0,1;\mathcal{D}((u\underline{\Delta}_{s})^{k+1}))},
\end{eqnarray*}
for appropriate constants $c_1,c_2$ and $c_3$ and $c>0$.
\end{proof}
 
\subsection{The domain of the bilaplacian}\label{bilaplacian}
Associated with the extension $\underline{\gD}_s$ of the Laplacian given in \eqref{dds} we define the
bilaplacian $\underline{\Delta}_{s}^{2}$ as the unbounded operator in $\mathcal{H}_{p}^{s,\gamma}(\mathbb{B})$ with domain
\begin{gather*}
\mathcal{D}(\underline{\Delta}_{s}^{2})=\{ u\in \mathcal{D}(\underline{\Delta}_{s}) \, | \, \underline{\Delta}_{s}u\in \mathcal{D}(\underline{\Delta}_{s})\}.
\end{gather*}
According to Lemma 4.3 in \cite{RS}, 
\begin{gather}\label{domdelta2}
\mathcal{D}(\underline{\Delta}_{s}^{2})
=\mathcal{H}_{p}^{s+4,\gamma+4}(\mathbb{B})
\oplus\bigoplus_{\rho}\widetilde{\mathcal{E}}_{\rho}\oplus\mathbb{C}.
\end{gather}
Here, the summation in the direct sum is over all poles of the inverse of the conormal symbol of 
$\underline{\Delta}_{s}^{2}$ such that the associated asymptotics space
$\widetilde{\mathcal{E}}_{\rho}$ is a subset of $\cD(\underline\gD_s)$. 
In more detail: The inverse of the conormal symbol 
has been computed in \cite[Section 4.2]{RS}:
\begin{eqnarray*}
\gs_M(\gD^2)(z)^{-1} &=&\sum_{j=0}^\infty\frac{1}{(z-q_j^+)(z-q_j^-)(z+2-q_j^+)(z+2-q_j^-)}\pi_j,
\end{eqnarray*}
where the $q^\pm_j$ are as before and $\pi_j$ is the $L^2$-orthogonal projection 
onto $E_j$.
The poles therefore lie in the points $\rho=q_j^\pm$ and $\rho=q_j^\pm-2$, $j=0,1,\ldots$. 
The asymptotics spaces have the form 
\begin{gather}\label{erho}
\widetilde{\mathcal{E}}_{\rho}=\{x^{-\rho}\log(x)\omega(x)e_{1}(y)+x^{-\rho}\omega(x)e_{0}(y)\}.
\end{gather}
Here, $\omega$ is a cutoff function and, for $\rho$ as above, $e_{0}$, $e_{1}$ belong to the eigenspace of the boundary Laplacian to the eigenvalue $\rho$; the $\log$-terms only occur if $\rho$ is a double pole. 

\subsection{Geometric Assumption.}
We will now suppose that 
\begin{eqnarray}\label{gl1}
\gl_1<\begin{cases} -9/4, & n=1\\ -n, &n\ge 3.\end{cases}. \end{eqnarray}
This implies that the value $\overline\gve_n$ introduced in Equation \eqref{epsilonbar} is 
larger than $3/2$ for $n=1$ and larger than $1$ for $n>2$. 

We can then restrict the choice of the weight $\gg$ further by asking that 
\begin{eqnarray}\label{choiceggrest1}
-\frac{1}2<\gg<\min\left\{\overline\gve_1-2, 0\right\} \text{ for } n=1,
\end{eqnarray} 
and 
\begin{eqnarray}\label{choiceggrest}
\frac{n-3}2<\gg<\min\left\{\frac{n-3}2+\overline\gve_n-1, \frac{n-1}2\right\} \text{ for } n\ge 3.
\end{eqnarray} 

This has an important consequence on the asymptotics spaces arising in the domain of the bilaplacian. Recall that a function $x^{-\rho}\log^kx\,e(y)$ for $0\not=e\in C^\infty(\partial \B)$ 
belongs to $\cH^{s,\gs}_p(\B)$ (for any choice of $s\in \R$ and $1<p<\infty$) if and only if
$\Re \rho <(n+1)/2-\gs.$
In order to have $\widetilde{\cE}_\rho\subset \cD(\underline\gD_s)$ we need to have 
$\widetilde\cE_\rho\subset \cH^{s+2,\gg+2}_p(\B)$ and thus 
$$\Re \rho <\frac{n+1}2-\gg-2= \frac{n-3}2-\gg.$$ 
On the other hand we know that the only values of $\rho$ that can arise are those of the form $q_j^\pm$ and $q_j^\pm-2$. This leads to the following statement.

\begin{lemma}\label{embedding}Let $n\not=2$. 
With the choice of the weight $\gg$ in \eqref{choiceggrest1}/\eqref{choiceggrest} we have 
$$\cD(\underline\gD_s^2)\subseteq \cH^{s+4,\frac{n+1}2+\min\{\overline\gve_n,2\}-\gve}_p(\B)\oplus \C \ \text{ for every } \ \gve>0.$$
Note that $\frac{n+1}2+\overline\gve_n>\gamma+3$ by the choice of $\gg$ in 
\eqref{choiceggrest1}/\eqref{choiceggrest}.
\end{lemma}
\begin{proof} 
In view of the specific form \eqref{domdelta2} of the domain, we have to show that there is no $\rho$ with $\Re\rho $ in the 
interval $(\max\{-\overline\gve_n,-2\},\frac{n-3}2-\gg]$. 

Let us consider first the case $n=1$, where $\overline \gve_1>3/2$. 
By \eqref{choiceggrest1}, 
$\frac{n-3}2-\gg<-1/2$. On the other hand, we have $q^\pm_0 = 0$, thus $q^\pm_0-2=-2$, $q_1^\pm= \pm\overline\gve_1$, 
hence $q^+_1-2>-1/2$. So none of the $\rho$ lies in the 
critical interval.

For $n\ge3$ we have $\overline \gve_n>1$ and $\frac{n-3}2-\gg<0$. 
As for the possible asymptotics, we have $q_0^- = 0$, so $q_0^--2=-2$ , 
$q^+_0=n-1$ and thus $q^+_0-2 = n-3\ge0$, $q^-_1=-\overline\gve_n<-1$. 
Again, none of the $\rho$ lies in the critical interval. 
\end{proof}

\begin{lemma}\label{inter2}
Let $n\not=2$ and the weight $\gg$ be chosen as in \eqref{choiceggrest1}/\eqref{choiceggrest}. Then
$$(\cD(\underline\gD_s^2),\cD(\underline\gD_s))_{\frac{1}{q},q}\hookrightarrow \cH^{s+4-\frac{2}{q}-\varepsilon,\sigma-\gve}_p(\B)\oplus\C, $$
for $\sigma = \frac{n+1}2+\min\{\overline\gve_n,2\}-\left(\frac{n-3}2+\min\{\overline\gve_n,2\}-\gamma\right)\frac1q$ and any $\varepsilon>0$.

As $\gg>\frac{n-3}2$ by \eqref{choiceggrest}, we see that 
$\gs>\frac{n+1}2+\min\{\overline\gve_n,2\}\left(1-\frac1q\right)>\frac{n+1}2$.
\end{lemma}
\begin{proof}
According to Lemma \ref{embedding} 
$$(\cD(\underline\gD_s^2),\cD(\underline\gD_s))_{\frac{1}{q},q}\hookrightarrow
(\cH^{s+4,\frac{n+1}2+\min\{\overline\gve_n,2\}-\gve}_p(\B)\oplus \C,\cH^{s+2,\gg+2}_p(\B)\oplus \C)_{\frac{1}{q},q}.$$
Interpolation for direct sums, see e.g. Proposition I.2.3.3 in \cite{Am}, yields the desired embedding. 
\end{proof}

\begin{proposition}\label{eqdm} 
Let $n\not=2$, 
and $\gamma$ satisfy \eqref{choiceggrest1}/\eqref{choiceggrest}. 
For every $u\in \cD(\underline\Delta_s)$ 
which is strictly positive on $\B$ we then have 
$\mathcal{D}((u\underline{\Delta}_{s})^{2})=\mathcal{D}(\underline{\Delta}_{s}^{2})$.
\end{proposition}
\begin{proof} 
In view of \eqref{scale.0}, we have $\mathcal{D}((u\underline{\Delta}_{s})^{2})=\{v\in \mathcal{D}(\underline{\Delta}_{s}) |\, u\underline{\Delta}_{s}v \in \mathcal{D}(\underline{\Delta}_{s})\}$. 
Since $u$ is strictly positive, Lemma \ref{linv} shows that also 
$u^{-1}\in \cH^{s+2,\gamma+2}_p(\B)\oplus\C$. 
As this is a Banach algebra, we conclude that 
$u\underline\Delta_sv\in \mathcal{D}(\underline{\Delta}_{s}) $ if and only if 
$\underline\Delta_sv\in \mathcal{D}(\underline{\Delta}_{s}) $. This shows the assertion. 
\end{proof}

\subsection{Analysis of the Porous Medium Equation} 
For $\gg$ satisfying \eqref{choiceggrest1}/\eqref{choiceggrest} let 
\begin{gather}\label{X0}
Y_{0}=\mathcal{H}_{p}^{s+2,\gamma+2}(\mathbb{B})\oplus\mathbb{C},
\end{gather}
equipped with the norm
\begin{gather*}
\|u_1+ c\|_{Y_{0}}=\|u_{1}\|_{\mathcal{H}_{p}^{s+2,\gamma+2}(\mathbb{B})}+|c|,
\quad u_{1}\in \mathcal{H}_{p}^{s+2,\gamma+2}(\mathbb{B}), c\in\mathbb{C}.
\end{gather*}
Lemma \ref{c0} implies that $Y_{0}$ is a Banach algebra (up to the choice of an equivalent norm). Moreover, $Y_0$ consists of bounded continuous functions, i.e. $Y_{0}\hookrightarrow C(\mathbb{B})$. 

In the sequel we shall use the closed extension $\underline{\Delta}$ of the Laplacian in $Y_{0}$ with domain
\begin{eqnarray}\label{ed2}
\mathcal{D}(\underline{\Delta})=Y_{1}:=\mathcal{D}(\underline{\Delta}_{s}^{2})
=\mathcal{H}_{p}^{s+4,\gamma+4}(\mathbb{B})
\oplus\bigoplus_{\rho}\widetilde{\mathcal{E}}_{\rho}\oplus\mathbb{C}
\end{eqnarray}
as defined in \eqref{domdelta2}. From the standard properties of operators in powers scales, see e.g. Lemma V.1.2.3 in \cite{Am}, and Theorem \ref{t11}, we obtain the following.
\begin{proposition}\label{t222}
For $c\notin \R_{\le0}$, $\theta\in[0,\pi)$ and $\phi>0$, we have $c-\underline{\Delta}\in \mathcal{P}(\theta)\cap\mathcal{BIP}(\phi)$. 
\end{proposition}

\begin{proposition}\label{p2} 
Let $n\not=2$, 
and $\gamma$ satisfy \eqref{choiceggrest1}/\eqref{choiceggrest}. For each $\gt\in[0,\pi)$ and each $u\in Y_{1/q,q}$ which is strictly positive on $\mathbb{B}$,
$c-u\underline{\Delta}: Y_1\rightarrow Y_0$ is $R$-sectorial of angle $\gt$ for sufficiently large $c>0$. In particular, $-u\underline\Delta$ has maximal $L^{q}$-regularity.
\end{proposition}
\begin{proof}
This is immediate from Proposition \ref{eqdm}, Lemma \ref{scale} and the fact that $Y_{1/q,q}\hookrightarrow X_{1/q,q}$.
\end{proof}

The restriction on $\lambda_{1}$ now allows us to control the action of $\nabla$ on the interpolation spaces:

\begin{lemma}\label{l1}Let $n\not=2 $, 
and $\gg$ be chosen according to \eqref{choiceggrest1}/\eqref{choiceggrest}. 
Assume in addition that $q$ is so large that 
$\gg<\overline\gve_1-2-1/(q-1)$ for $n=1$ and 
$\gg<(n-3)/2+\overline\gve_n -1-1/(q-1)$ for $n\ge3$. 
Then, in local coordinates near the boundary, the operators $\partial_{x}$ and $\frac{1}{x}\partial_{y_{i}}$, $i\in\{1,...,n\}$, induce bounded maps from 
$Y_{\frac{1}{q},q}:=(Y_{1},Y_{0})_{\frac{1}{q},q}$ to $Y_{0}$.
\end{lemma}
\begin{proof}
According to Lemma \ref{inter2}, $\frac1x\partial_{y_j}$ maps $Y_{1/q,q}$ to 
$\mathcal{H}_{p}^{s+3-\frac{2}{q}-\varepsilon,\gs-1-\varepsilon}(\mathbb{B})$ 
continuously for every $\gve>0$, with $\gs$ defined there.
Our assumption \eqref{spq} implies that $q>2$ so that $s+3-2/q>s+2$. In order to establish the assertion it suffices to check that $\gs-1>\gg+2$. A short computation shows that our assumptions guarantee precisely this. 
The operator $\partial _{x}$ can be treated in the same way. 
\end{proof}

We are now ready to state the main result of this section. For better legibility we repeat the assumptions. 

\begin{theorem}
Let $\dim (\B)\not=3$ and assume that the first eigenvalue $\gl_1$ of the boundary Laplacian satisfies condition \eqref{gl1}. Let $s>-1$ and let the weight $\gg$ satisfy \eqref{choiceggrest1}/\eqref{choiceggrest}. Choose $1<p,q<\infty$ so large that 
\begin{eqnarray*}
-1+\frac{n+1}p+\frac2q<\min\{0,s\} \ \text{and} \ 
\gg<\begin{cases}\overline\gve_1-2-\frac{1}{q-1}; & n=1\\
\frac{n-3}2+\overline\gve_n -1-\frac1{q-1};& n\ge3.
\end{cases}
\end{eqnarray*}
Denote by $\underline{\Delta}: \cD(\underline\Delta)=Y_1\to Y_0$ the closed extension of the Laplacian in $Y_0$ associated with these data, defined in \eqref{ed2}.

For every strictly positive initial value $u_0\in Y_{1/q,q}$ 
we then find some $0<T\leq T_{0}$ such that the porous medium equation \eqref{e1}, \eqref{e2}, considered in the space $L^{q}(0,T;Y_0)$, has a unique solution 
$$
u\in L^{q}(0,T;Y_1)\cap W^{1,q}(0,T;Y_0)).
$$
\end{theorem} 
\begin{proof} 
We only have to check that Cl\'ement and Li's Theorem \ref{CL} can be applied to \eqref{e3}, \eqref{e4} with the Banach couple $Y_{0}$, $Y_{1}$.
Since $Y_{1/q,q}$ embeds to $\cH^{s+2,\gamma+2}_p(\B)\oplus\C$, Lemma \ref{linv} together with Proposition \ref{p2} shows that the operator $-m u_{0}^{m-1}\underline{\Delta}$ has maximal $L^q$-regularity 
for any strictly positive $u_{0}\in Y_{1/q,q}$. 

Similarly as in Equation \eqref{h1}, property (H1) follows from the fact that $Y_{1/q,q}$ embeds into a Banach algebra of multipliers on $Y_0$. 
Compared to Theorem \ref{pmt}, the proof of (H2) even simplifies. Since $Y_0$ is a 
Banach algebra, Lemma \ref{l1} implies the estimate
$$\|\langle \nabla v_1, \nabla v_2\rangle_g\|_{Y_0} 
\le C\|v_1\|_{Y_{1/q,q}}\|v_2\|_{Y_{1/q,q}}, \quad v_1,v_2\in Y_{1/q,q}.$$
Estimate \eqref{pmt.1} (with $X_0$ replaced by $Y_0$) can then be continued as follows 
\begin{eqnarray*}
&\le & C_1 \|u_1-u_2\|_{Y_0} + C_2 |t_1-t_2|+ C_3\|u_1-u_2\|_{Y_{1/q,q}}. 
\end{eqnarray*}
This completes the argument.
\end{proof}


\begin{thebibliography}{99}

\bibitem{Am1} H. Amann. {\em Function spaces on singular manifolds}, Math. Nachr. {\bf 286}, no. 5-6, 436--475 (2013). 

\bibitem{Am} H. Amann. {\em Linear and quasilinear parabolic problems Vol. I. Abstract linear theory}. Monographs in Mathematics {\bf 89}, Birkh\"auser Verlag (1995).
 
\bibitem{BG} M. Bonforte, G. Grillo. {\em Asymptotics of the porous media equation via Sobolev inequalities}. J. Funct. Anal. {\bf 225}, no. 1, 33--62 (2005).

\bibitem{BS} G. Bourdaud, W. Sickel. {\em Composition operators on function spaces with fractional order of smoothness}. RIMS Kokyuroku Bessatsu {\bf B26}, 93--132 (2011).

\bibitem{BrueningSeeley88} J. Br\"uning, R. Seeley. {\em An index theorem for first order regular singular operators}. Amer. J. Math. {\bf 110}, 659--714 (1988).
 
\bibitem{CL} P. Cl\'ement, S. Li. {\em Abstract parabolic quasilinear equations and application to a groundwater flow problem}. Adv. Math. Sci. Appl. {\bf 3}, Special Issue, 17--32 (1993/94).

\bibitem{CP} P.\ Cl\'ement, J.\ Pr\"uss. {\em An operator-valued transference principle and maximal regularity on vector-valued $L_p$-spaces}. In: G. Lumer and L. Weis (eds.), Proc. of the 6th. International Conference on Evolution Equations, Marcel Dekker (2001).

\bibitem{CSS2} S. Coriasco, E. Schrohe, J. Seiler. {\em Bounded imaginary powers of differential operators on manifolds with conical singularities}. Math. Z. {\bf 244}, 235--269 (2003).

\bibitem{CSS1} S. Coriasco, E. Schrohe, J. Seiler. {Differential operators on conic manifolds: Maximal regularity and parabolic equations}. Bull. Soc. Roy. Sci. Li\`ege {\bf 70}, no. 4-6, 207--229 (2001). 

\bibitem{DHP} R. Denk, M. Hieber, J. Pr\"uss. {\em $R$-boundedness, Fourier multipliers, and problems of elliptic and parabolic type}. Memoirs of the American Mathematical Society {\bf 166}, no. 788, Oxford University Press (2003).

\bibitem{Do} G.\ Dore. {\em $L^{p}$ regularity for abstract differential equations}. In H. Komatsu (editor), Functional Analysis and related topics, Lect. Notes in Math. {\bf 1540}, Springer Verlag (1993). 

\bibitem{DV} G. Dore, A. Venni. {On the closedness of the sum of two closed operators}. Math. Z. {\bf 196}, 189--201 (1987). 

\bibitem{GKM} J. Gil, T. Krainer, G. Mendoza. {\em Resolvents of elliptic cone operators}. J. Funct. Anal. {\bf 241}, no. 1, 1--55 (2006).

\bibitem{Hua} G. Huang, Z. Huang, H. Li. {\em Gradient estimates for the porous medium equations on Riemannian manifolds}. J. Geom. Anal. {\bf 23}, no. 4, 1851--1875 (2013). 

\bibitem{KS} M. Kaip, J. Saal. {\em The permanence of R-boundedness and property $(\alpha)$ under interpolation and applications to parabolic systems}, J. Math. Sci. Univ. Tokyo {\bf 19}, no. 3, 359--407 (2012).

\bibitem{KW1} N. Kalton, L. Weis. {\em The $H^{\infty}$-calculus and sums of closed operators}. Math. Ann. {\bf 321}, no. 2, 319--345 (2001).

\bibitem{KL1} P. C. Kunstmann, L. Weis. {\em Maximal $L_p$-regularity for parabolic equations, Fourier multiplier theorems and $H^{\infty}$-functional calculus}. Functional Analytic Methods for Evolution Equations, Lecture Notes in Mathematics {\bf 1855}, Springer Verlag, 65--311 (2004).

\bibitem{KL} P. C. Kunstmann, L. Weis. {\em Perturbation theorems for maximal $L_p$-regularity}. Ann. Scuola Norm. Sup. Pisa Cl. Sci. (4) {\bf 30}, no. 2, 415--435 (2001).

\bibitem{Le} M. Lesch. {\em Operators of Fuchs type, conical singularities, and asymptotic methods}. Teubner-Texte zur Mathematik {\bf 136}, Teubner Verlag (1997). 

\bibitem{MRS} R. Mazzeo, Y. Rubinstein, N. Sesum. {\em Ricci flow on surfaces with conic singularities}. Anal. PDE {\bf 8}, no. 4, 839--882 (2015).

\bibitem{Ott} F. Otto. {\em The geometry of dissipative evolution equations: the porous medium equation}. Comm. Partial Differential Equations {\bf26}, no. 1-2, 10--174 (2001). 

\bibitem{Ro1} N. Roidos. {\em On the inverse of the sum of two sectorial operators}. J. Funct. Anal. {\bf 265}, no. 2, 208--222 (2013).

\bibitem{Ro2} N. Roidos. {\em Preserving closedness of operators under summation}. J. Funct. Anal. {\bf 266}, no. 12, 6938--6953 (2014). 

\bibitem{RS0} N. Roidos, E. Schrohe. {\em The Cahn-Hilliard equation and the Allen-Cahn equation on manifolds with conical singularities}. Comm. Partial Differential Equations {\bf 38}, no. 5, 925--943 (2013).

\bibitem{RS} N. Roidos, E. Schrohe. {\em Bounded imaginary powers of cone differential operators on higher order Mellin-Sobolev spaces and applications to the Cahn-Hilliard equation}. J. Differential Equations {\bf257}, no. 3, 611--637 (2014).

\bibitem{ScSe} E. Schrohe, J. Seiler. {\em Ellipticity and invertibility in the cone algebra on $L_p$-Sobolev spaces}. Integral Equations Operator Theory {\bf 41}, no. 1, 93--114 (2001).

\bibitem{Sh} E. Schrohe, J. Seiler. {\em The resolvent of closed extensions of cone differential operators}. Canad. J. Math. {\bf 57}, no. 4, 771--811 (2005).

\bibitem{se} J. Seiler. {\em Parameter-dependent pseudodifferential operators of Toeplitz type}. Annali di Matematica Pura ed Applicata {\bf 194}, no. 1, 145--165 (2015).

\bibitem{Shao} Y. Shao. {\em Singular parabolic equations of second order on manifolds with singularities}. J. Differential Equations {\bf 260}, no. 2, 1747--1800 (2016).

\bibitem{SW} E. M. Stein, G. Weiss. {\em Introduction to Fourier analysis on euclidean spaces}. Princeton Mathematical Series {\em 32}, Princeton University Press (1971).

\bibitem{Tay1} M. Taylor. {\em Partial Differential Equations I, Basic Theory}. Applied Mathematical Sciences {\bf 117}, Springer Verlag (1997).

\bibitem{Tay} M. Taylor. {\em Partial Differential Equations III, Nonlinear Equations}. Applied Mathematical Sciences {\bf 117}, Springer Verlag (2011).

\bibitem{W} L. Weis. {\em Operator-valued Fourier multiplier theorems and maximal $L_{p}$-regularity}. Math. Ann. {\bf 319}, no. 4, 735--758 (2001). 

\bibitem{Zh} Q. Zhang. {\em Blow-up results for nonlinear parabolic equations on manifolds}. Duke Math. J. {\bf 97}, no. 3, 515--539 (1999). 

\bibitem{Zhu} X. Zhu. {\em Hamilton's gradient estimates and Liouville theorems for porous medium equations on noncompact Riemannian manifolds}. J. Math. Anal. Appl. {\bf 402}, no. 1, 201--206 (2013). 

\end{thebibliography}
\end{document}